\documentclass[11pt]{amsart}
\usepackage[utf8]{inputenc}

\usepackage{amssymb, mathtools}
\usepackage[margin=1in]{geometry}
\usepackage{mathrsfs}
\usepackage{hyperref}
\usepackage{textgreek}

\newtheorem{theorem}{Theorem}[section]
\newtheorem{proposition}[theorem]{Proposition}
\newtheorem{corollary}[theorem]{Corollary}
\newtheorem{lemma}[theorem]{Lemma}
\newtheorem{question}[theorem]{Question}

\theoremstyle{definition}
\newtheorem{definition}[theorem]{Definition}

\allowdisplaybreaks

\begin{document}

\title{Equidistribution in shrinking sets for arithmetic spherical harmonics}
\author{Maximiliano Sanchez Garza}
\date{}

\begin{abstract}
We study a variant of the equidistribution of mass conjecture on the sphere posed by Böcherer, Sarnak, and Schulze-Pillot: quantum unique ergodicity in shrinking sets. Conditionally on the generalized Lindelöf hypothesis, we show that quantum unique ergodicity holds on every shrinking spherical cap whose radius is considerably larger than the Planck scale, and that it holds on almost every shrinking spherical cap whose radius is larger than the Planck scale. Additionally, conditionally on GLH, we provide explicit upper bounds for the $1$-Wasserstein distance and the spherical cap discrepancy between the involved measures.
\end{abstract}

\maketitle

\section{Introduction}\label{Introduction}

Let $(M,g)$ be a compact and connected Riemannian manifold and let $\{\phi_n\}$ be an orthonormal basis of $L^2(M)$ consisting of Laplacian eigenfunctions. Let $\mu$ be the associated measure to $g$ on $M$, normalized so that it is a probability measure. By definition, we know that $|\phi_n|^2 \: \mathrm{d}\mu$ is a probability measure on $M$. One may ask how this probability measure behaves when $n\rightarrow \infty$ (or, equivalently, as the Laplacian eigenvalue goes to infinity).

We say that the basis $\{\phi_n\}$ is \textit{quantum ergodic} if there exists a subsequence $\{\phi_{m_n}\}$ of $\{\phi_n\}$ of density $1$ such that the sequence of probability measures $\{|\phi_{m_n}|^2 \: \mathrm{d}\mu\}$ converges in distribution to the probability measure $\mathrm{d}\mu$; that is, $\{\phi_n\}$ is quantum ergodic if there exists a subsequence $\{\phi_{m_n}\}$ of $\{\phi_n\}$ of density $1$ such that for every continuity set\footnote{A subset $B\subset M$ is a \textit{continuity set} if it is measurable and its boundary has measure $0$.} $B\subset M$, we have $$ \int_{B} |\phi_{m_n}|^2 \: \mathrm{d}\mu \longrightarrow \mathrm{vol}(B)$$ as $n\rightarrow \infty$. We say that $\{\phi_n\}$ is \textit{quantum unique ergodic} (or that it satisfies QUE) if for every continuity set $B\subset M$, we have 
\begin{equation}
    \int_{B} |\phi_{n}|^2 \: \mathrm{d}\mu \longrightarrow \mathrm{vol}(B) \label{QUEDef}
\end{equation}
as $n\rightarrow \infty$.

The previous notion of QUE is known as \textit{quantum unique ergodicity in configuration space}. There is a stronger version of QUE, called \textit{quantum unique ergodicity in phase space}, where one instead asks for the associated microlocal lifts of these measures to the unit cotangent bundle $S^{*}M$ to equidistribute with respect to the Liouville measure. Here, we will only concern ourselves with QUE in configuration space; for more on QUE in phase space, see \cite{SarnakProgressQUE}.

A foundational result in this context is due to Shnirel'man, Colin de Verdière, and Zelditch (see \cite{ShnirelmanQE}, \cite{deVerdiere85}, \cite{ZelditchQE}, \cite{Shnirelman1993}), which states that if the geodesic flow in $S^{*}M$ is ergodic, then \textit{any} orthonormal basis of $L^2(M)$ is quantum ergodic. However, Rudnick and Sarnak in \cite{RudnickSarnakQUE} conjectured that, in the case that $M$ is a compact Riemannian manifold of negative curvature, any orthonormal basis should be quantum unique ergodic. In the case of compact hyperbolic surfaces arising from quaternion algebras and the orthonormal basis consisting of eigenfunctions of the Laplacian operator and all the Hecke operators, this conjecture was proven by Watson in \cite[\S 4.2]{WatsonPhD} assuming the generalized Lindelöf hypothesis (GLH) and by Lindenstrauss \cite[Thm.~1.4]{LindenstraussQUE} unconditionally.

Even though this discussion is in the context of compact Riemannian manifolds, similar conjectures can be made in case that the manifold is noncompact but has finite volume. In particular, for $M=\mathrm{SL}_{2}(\mathbb{Z}) \backslash \mathbb{H}$, one has to consider a continuous spectrum and a discrete spectrum of the Laplacian operator. It is known that the continuous spectrum satisfies a version of quantum unique ergodicity by work of Jakobson in \cite{JakobsonQUEEisenstein} and of Luo and Sarnak in \cite{LuoSarnakQUE}. For the discrete spectrum, the basis of Hecke--Maass cusp forms is known to be quantum unique ergodic conditionally to GLH via work of Watson in \cite[\S 4.2]{WatsonPhD}. Later, Lindenstrauss in \cite{LindenstraussQUE} almost showed the result unconditionally, but could not eliminate the possibility of escape of mass at the cusp at infinity. This possibility was eliminated by Soundararajan in \cite{SoundQUE}, so that the basis of Hecke--Maass cusp forms was shown unconditionally to be quantum unique ergodic. This version of QUE is often referred to as \textit{arithmetic quantum unique ergodicity}. 

There are three remarks in order. The first one, as Soundararajan mentions in \cite[pp.~1530--1531]{SoundQUE}, is that his proof of arithmetic QUE works not only for $\mathrm{SL}_2(\mathbb{Z})$, but for \textit{any} congruence subgroup. Hence, arithmetic QUE is known for any noncompact arithmetic quotient of $\mathbb{H}$. The second remark is that, since it is believed that all the Laplacian eigenvalues in $\mathrm{SL}_{2}(\mathbb{Z}) \backslash \mathbb{H}$ are simple, the additional requirement on the Laplacian eigenfunctions to be eigenfunctions of all the Hecke operators would be a vacuous condition and we would obtain QUE automatically in this case. The third remark is that both Lindenstrauss' and Soundararajan's results are only qualitative, i.e.~they do not quantify the rate of convergence of \eqref{QUEDef}. However, the proof of Watson in \cite{WatsonPhD} shows that the optimal rates of convergence follow from GLH. For a more extensive discussion on QUE and its arithmetic version, see \cite{SarnakProgressQUE}.

\subsection{Quantum ergodicity on the sphere}

Our main focus for the rest of this paper is the $2$-sphere, $M=S^2$. We give it the round metric and associated measure, which gives it the structure of a Riemannian manifold. Let $\Delta$ be its Laplace operator. We know that $L^2(S^2)$ has a complete spectral resolution with respect to $\Delta$. Moreover, this spectral resolution is completely explicit in terms of spherical harmonics; see \cite[Thm.~2.9]{SpectrumLaplacian}, \cite[p.~140]{SteinWeiss}, \cite[Thm.~2.34, Prop.~3.5]{SpheHar}. More explicitly, the $\ell$-th Laplacian eigenvalue is $\lambda_{\ell} = \ell(\ell+1)$ and its eigenspace, denoted by $H_{\ell}$, is the space of harmonic homogeneous polynomials in three variables of degree $\ell$ restricted to the sphere $S^2$, which are called \textit{spherical harmonics}. Whenever $\phi$ is a spherical harmonic, we will denote its degree by $\ell_{\phi}$ (or just $\ell$ if there is no confusion).

First, since the geodesic flow of the sphere is not ergodic, the aforementioned results do not apply. In fact, quantum ergodicity (and hence QUE) is not expected for all bases of $L^2(S^2)$. Indeed, the typical basis of zonal, tesseral, and sectorial spherical harmonics is known to \textit{not} be quantum ergodic, since a positive-density subsequence of this basis concentrates on the equator of $S^2$ (see \cite[pp.~497--498]{deVerdiere85}). However, since the eigenspaces have dimension bigger than one, it makes sense to consider other bases to see if equidistribution occurs. In this direction, Zelditch in \cite{ZelditchQES2} showed that almost every orthonormal basis of $L^2(S^2)$ composed of Laplacian eigenfunctions is quantum ergodic\footnote{Zelditch showed that the space $X$ of all orthonormal bases of $L^2(S^2)$ composed of Laplacian eigenfunctions is compact and has a group structure, so that it possesses a Haar probability measure. Thus, here ``almost'' means that there exists a measure zero subset $Y$ of $X$ (with respect to the Haar measure) such that every orthonormal basis in $X \setminus Y$ is quantum ergodic.}. Later, VanderKam in \cite[Thm.~1.3]{VanderKam} improved Zelditch's result by proving that almost every orthonormal basis of $L^2(S^2)$ composed of Laplacian eigenfunctions is quantum unique ergodic. Some explicit examples of quantum ergodic bases of $L^2(S^2)$ can be constructed via averaging operators; see \cite{BLLQE} for more details.

Similar to the case of hyperbolic surfaces, one can consider the arithmetic case for $S^2$ by diagonalizing a basis of $L^2(S^2)$ with respect to the Hecke operators. We describe now how these operators are constructed; for more details, see \cite{garza2025geodesicrestrictionproblemarithmetic}. 

For any ring $R$, we define ring of Hamiltonian quaternions over $R$ by $$B(R) := \langle 1,i,j,k \: | \: i^2=j^2=k^2=-1, \: ij=-ji=k\rangle_R,$$ so that $B(R) = B(\mathbb{Q}) \otimes_{\mathbb{Q}} R$ for any $\mathbb{Q}$-algebra $R$. We denote by $B^{\times}(R)$ the group of units of $B(R)$. Then $B^{\times}$ is an affine algebraic group over $\mathbb{Q}$. For $\alpha = a+bi+cj+dk \in B(R)$, let $\mathrm{Nrd}(\alpha) := a^2+b^2+c^2+d^2 \in R$ and $\mathrm{Trd}(\alpha) := 2a \in R$ be its (reduced) norm and (reduced) trace, respectively. Let $B^{0,1}(R)$ be the subset of $B(R)$ of quaternions of trace zero and norm $1$. The group $B^{\times}(R)\ni \alpha$ acts on $B^{0,1}(R) \ni \gamma$ by conjugation, i.e.~$\alpha \cdot \gamma := \alpha \gamma \alpha^{-1}$. 

Specializing to the case $R=\mathbb{R}$, we have the bijection\footnote{Throughout the paper, we identify $B^{0,1}(\mathbb{R})$ and $S^2$ under the above bijection and use them interchangeably.}
\begin{align*}
    B^{0,1}(\mathbb{R}) &\longrightarrow S^2 \\
    xi+yj+zk &\longmapsto (x,y,z),
\end{align*}
which gives an action of $B^{\times}(\mathbb{R})$ on $S^2$. This action is by rotations (see \cite[\S 3.1]{ConwaySmith}). Hence, there is an induced action on $L^2(S^2)$, which is the left regular representation of $B^{\times}(\mathbb{R})$ in $L^2(S^2)$. 

Let $$\mathcal{O} := \left\langle 1,i,j,\frac{1+i+j+k}{2}\right\rangle_{\mathbb{Z}}$$ be the ring of Hurwitz integers, which is a maximal order in $B(\mathbb{R})$ of class number $1$. If $\mathcal{O}^{\times}$ denotes the group of units of $\mathcal{O}$, we have that $|\mathcal{O}^{\times}| = 24$. For any positive integer $n$, we let $$ \mathcal{O}(n) := \left\{ \alpha \in \mathcal{O} : \mathrm{Nrd}(\alpha) = n \right\}$$ and define the $n$-th Hecke operator $T_n : C(S^2) \rightarrow C(S^2)$ by $$ (T_n f)(z) := \frac{1}{|\mathcal{O}^{\times}|}\sum_{\gamma \in \mathcal{O}(n)} (\gamma^{-1} \cdot f)(z) = \frac{1}{|\mathcal{O}^{\times}|}\sum_{\gamma \in \mathcal{O}(n)} f(\gamma \cdot z),$$ where $\gamma^{-1}$ is interpreted as an element of $B^{\times}(\mathbb{R})$. Since the action of $T_{2^n}$ is that of $T_1$ or $T_2$ depending on the $2$-adic valuation of $n$, the Hecke operators of interest are the ones corresponding to odd values of $n$. We restrict to these for the present discussion.

Since the Laplacian is rotation-invariant, all the Hecke operators commute with $\Delta$. Moreover, one can show that all the Hecke operators commute with each other and are self-adjoint in $C(S^2)$. Hence, if $H_{\ell}^{\mathcal{O}^{\times}}$ denotes the subspace of $H_{\ell}$ of functions that are invariant under the action of $\mathcal{O}^{\times}$, we can diagonalize a basis of each $H_{\ell}^{\mathcal{O}^{\times}}$ to obtain an orthonormal basis $\mathcal{B}_{\ell}$ of spherical harmonics that are eigenfunctions of all the Hecke operators. We call these \textit{Hecke eigenfunctions}. These are elements of $L^2(\mathcal{O}^{\times} \backslash S^2)$, which is a Hilbert space with inner product
\begin{equation}
    \langle f,g \rangle := \int_{\mathcal{O}^{\times} \backslash S^2} f(x) \overline{g(x)} \: \mathrm{d}\sigma(x), \label{innerProdbigoS2}
\end{equation}
where $\sigma$ denotes the round measure on $S^2$, which satisfies $\sigma(S^2) = 4\pi$ and $\sigma(\mathcal{O}^{\times} \backslash S^2) = \pi/3$. We let $\mathcal{B} := \bigcup_{\ell\geq 1} \mathcal{B}_{\ell}$.

The Jacquet--Langlands correspondence provides a connection between automorphic representations of $B^{\times}$ and of $\mathrm{GL}_{2}$, which at the level of automorphic forms on the corresponding homogeneous spaces gives an explicit correspondence between spherical harmonics and modular forms. Indeed, if $S_{\ell}^{\text{new}}(\Gamma_0(2))$ denotes the space of holomorphic newforms of weight $\ell$ and level $2$, then the theory of theta series (see \cite[pp.~103--105]{EichlerHeckeOp}, \cite[Ch.~9,10]{Iwaniec}) provides a map
\begin{align}
    H_{\ell}^{\mathcal{O}^{\times}} &\longrightarrow S_{2\ell+2}^{\text{new}}(\Gamma_0(2)) \label{LiftSpherHarmToModForms} \\
    \psi &\longmapsto f_{\psi} \nonumber
\end{align}
given by 
\begin{equation}
    f_{\psi} (z) := \frac{1}{24}\sum_{\gamma \in \mathcal{O}} \left\langle \psi(\xi), \psi\left(\gamma \xi \overline{\gamma} \right) \right\rangle_{L^2(S^2)} e^{2\pi i \mathrm{Nrd}(\gamma) z}, \label{JLModForm}
\end{equation}
where $\left\langle \psi(\xi), \psi\left(\gamma \xi \overline{\gamma} \right) \right\rangle_{L^2(S^2)}$ denotes the inner product in $L^2(S^2)$. This map is an explicit realization of the aforementioned Jacquet--Langlands correspondence. Additionally, this correspondence tells us that the Hecke eigenspaces in $H_{\ell}^{\mathcal{O}^{\times}}$ are one-dimensional.

With this in mind, Böcherer, Sarnak, and Schulze-Pillot in \cite[Conj.~1]{BSP} conjectured that the basis $\{\psi_n\}$ of Hecke eigenfunctions of $L^2(\mathcal{O}^{\times} \backslash S^2)$ is quantum unique ergodic. Moreover, they showed that the conjecture would follow from subconvex estimates for degree $8$ $L$-functions (in particular, it would follow from GLH) by reducing the problem to proving that
\begin{equation}
    \int_{S^2} \phi(x) |\phi_n(x)|^2 \: \mathrm{d}\sigma(x) \rightarrow 0 \label{QUEAltDef}
\end{equation}
as $n\rightarrow \infty$, where $\phi$ is any Hecke eigenfunction. These subconvex estimates would provide rates of convergence for \eqref{QUEAltDef}.

Here, we look to provide explicit rates of convergence of the involved measures. To be more precise, for each Hecke eigenfunction $\psi$ of $S^2$ of Laplacian eigenvalue $\lambda_{\psi} = \ell_{\psi} (\ell_{\psi}+1)$, normalized such that $\langle \psi, \psi \rangle = 1$, we define the probability measure $\sigma_{\psi}$ of $\mathcal{O}^{\times} \backslash S^2$ via $$ \int_{\mathcal{O}^{\times} \backslash S^2} f(x) \: \mathrm{d}\sigma_{\psi}(x) := \int_{\mathcal{O}^{\times} \backslash S^2} f(x) |\psi(x)|^2 \: \mathrm{d}\sigma(x) $$ for all continuous functions $f:\mathcal{O}^{\times} \backslash S^2 \rightarrow \mathbb{C}$. Arithmetic QUE for the sphere then states that
\begin{equation}
    \sigma_{\psi} \rightarrow \frac{1}{\sigma(\mathcal{O}^{\times} \backslash S^2)}\sigma \label{QUEMeasures}
\end{equation}
in the weak-$*$ topology as $\ell_{\psi} \rightarrow \infty$. 

To provide explicit rates of convergence in \eqref{QUEMeasures}, we use the $1$-Wasserstein distance. Given two Borel probability measures $\mu_1, \mu_2$ on a Polish space $(X,d)$, we define the $1$-Wasserstein distance between $\mu_1$ and $\mu_2$ as $$ \mathscr{W}_1(\mu_1, \mu_2) := \inf_{\pi \in \Pi(\mu_1,\mu_2)} \int_{X\times X} d(x,y) \: \mathrm{d}\pi(x,y), $$ where $\Pi(\mu_1, \mu_2)$ is the set of couplings of $\mu_1$ and $\mu_2$, that is, the set of Borel probability measures $\pi$ on $X\times X$ such that $\pi(B\times X)=\mu_1(B)$ and $\pi(X\times B) = \mu_2(B)$ for every Borel set $B\subseteq X$. The $1$-Wasserstein distance plays an important role in optimal transport; see \cite{VillaniOptTransport} for more details. In our context, the $1$-Wasserstein distance defines a metric in the space of Borel probability measures $\mu$ on $X$ for which $\int_{X} d(x_0, x) \: \mathrm{d}\mu(x)$ is finite for some (and hence for all) $x_0 \in X$. Convergence with respect to the $1$-Wasserstein distance is basically convergence in the weak-$*$ topology.

Although the $1$-Wasserstein distance between two Borel probability measures is not easy to estimate in general, one can do so in certain special situations via the Berry--Esseen inequality, which gives an estimate of this distance in terms of the Fourier transform of the given measures. Using this inequality in the case of the $2$-sphere, we are able to give a rate of convergence in \eqref{QUEMeasures}.

\begin{theorem}\label{MainThm-Wass}
Let $\psi$ be a Hecke eigenfunction of Laplacian eigenvalue $\lambda_{\psi}^2 = \ell_{\psi}(\ell_{\psi}+1)$ for some integer $\ell_{\psi} \geq 0$, normalized such that $\langle \psi, \psi \rangle = 1$. Then, assuming GLH, for all $\varepsilon>0$ we have $$ \mathscr{W}_{1}\left(\sigma_{\psi}, \frac{1}{\sigma(\mathcal{O}^{\times} \backslash S^2)} \sigma \right) \ll_{\varepsilon} \ell_{\psi}^{-\frac{1}{2} + \varepsilon}.$$
\end{theorem}

The analogous result to Theorem \ref{MainThm-Wass} for $\mathrm{SL}_{2}(\mathbb{Z}) \backslash \mathbb{H}$ was recently shown by Humphries in \cite{Humphries2025} with the same rate of convergence, assuming GLH as well.

A stronger variant of quantum unique ergodicity is a shrinking target analog in which we allow the continuity set $B$ to vary with $\lambda_{\psi}$ and we ask if equidistribution still occurs. One can understand this problem as trying to capture the smallest scale at which the behavior of the arithmetic spherical harmonics remains random. The analogous question for $\mathrm{SL}_{2}(\mathbb{Z}) \backslash \mathbb{H}$ has been extensively studied; see \cite{HumphriesPhD} and \cite{YoungQUEShrinking} for a recount of the results related to this question and optimal convergence rates.

We specialize to the case in which $B$ is a spherical cap. For $x,y\in S^2$, if $d(x,y)$ denotes the distance between $x$ and $y$ in $S^2$, then $$d(x,y) = \arccos{(x\cdot y)},$$ where $x\cdot y$ denotes dot product in $\mathbb{R}^3$. Hence, for $w\in S^2$ and $0<R\leq \pi$, if $B_R(w)$ denotes the ball in $S^2$ centered at $w$ of radius $R$, we have $$ \mathrm{vol}(B_R(w)) = 4\pi \sin^2(R/2),$$ which is independent of the center $w$. Thus, we denote $\mathrm{vol}(B_R) := \mathrm{vol}(B_R(w))$. In particular, note that $\mathrm{vol}(B_R) \asymp R^2$ for $R$ sufficiently small. We seek to answer the following question.

\begin{question}\label{QUEQuestion}
Let $\psi$ be a Hecke eigenfunction of Laplacian eigenvalue $\lambda_{\psi}^2 = \ell_{\psi}(\ell_{\psi}+1)$ for some integer $\ell_{\psi}\geq 0$, normalized such that $\langle \psi, \psi\rangle = 1$. For what conditions on $R$, with regards to $\ell_{\psi}$, is it still true that 
\begin{equation}
    \frac{1}{\mathrm{vol}(B_R)} \int_{B_R(w)} |\psi(z)|^2 \: \mathrm{d}\sigma(z) = \frac{1}{\mathrm{vol}(\mathcal{O}^{\times}\backslash S^2)} + o_w(1) \label{QUESphericalCapQuestion}
\end{equation}
as $\ell_{\psi}$ tends to infinity? 
\end{question}

As remarked by Humphries in \cite[\S 1.3]{HumphriesPhD} (based on \cite[\S 5]{HejhalRackner}), thinking of $\ell_{\psi}^{-1}$ as the Planck scale, one should not expect equidistribution to hold when $R$ is below or within a logarithmic window of the Planck scale. Hence, we explore the case when $R$ is substantially bigger than $\ell_{\psi}^{-1}$, that is, $R\gg \ell_{\psi}^{-\delta}$ for $0<\delta <1$.

Question \ref{QUEQuestion} has been studied in the case of the modular surface by Young in \cite[Prop.~5]{YoungQUEShrinking}, where he showed, assuming GLH, that equidistribution still occurs on every shrinking ball of radius $R\gg t_{g}^{-\delta}$ where $g$ is a Hecke--Maass form of level $1$ of Laplacian eigenvalue $\frac{1}{4} + t_{g}^2$ and $0<\delta < \frac{1}{3}$. We prove an analogous result in our context.

\begin{theorem}\label{MainThm-AllCenters}
Let $\psi$ be a Hecke eigenfunction of Laplacian eigenvalue $\lambda_{\psi}^2 = \ell_{\psi}(\ell_{\psi}+1)$ for some integer $\ell_{\psi} \geq 0$, normalized such that $\langle \psi, \psi \rangle = 1$. Assume GLH and suppose that $R\gg \ell_{\psi}^{-\delta}$ for some $0<\delta < \frac{1}{3}$. Then, as $\ell_{\psi} \rightarrow \infty$, we have $$ \frac{1}{\mathrm{vol}(B_R)} \int_{B_R(w)} |\psi(z)|^2 \: \mathrm{d}\sigma(z) = \frac{1}{\mathrm{vol}(\mathcal{O}^{\times}\backslash S^2)} + o_{\delta}(1)$$ for every fixed point $w\in \mathcal{O}^{\times} \backslash S^2$.
\end{theorem}

Following the same ideas for the proof of Theorem \ref{MainThm-AllCenters}, we also obtain bounds on the spherical cap discrepancy between $\sigma_{\psi}$ and $\frac{1}{\sigma(\mathcal{O}^{\times} \backslash S^2)} \sigma$. More precisely, let $$ \mathcal{D}_{\mathrm{cap}}\left(\sigma_{\psi},\frac{1}{\sigma(\mathcal{O}^{\times} \backslash S^2)} \sigma \right) := \sup_{B} \left|\sigma_{\psi}(B) - \frac{\sigma(B)}{\sigma(\mathcal{O}^{\times} \backslash S^2)}\right|,$$ where the supremum ranges over all geodesic spherical caps\footnote{A spherical cap is \textit{geodesic} if it does not intersect itself; in other words, $B_{R}(w)$ is a geodesic spherical cap if for all $z\in B_{R}(w)$, we have that $d(z,w)<R$ and the distance from $z$ to $w$ along any other geodesic path is $\geq R$.} $B$ of $\mathcal{O}^{\times} \backslash S^2$. This is yet another quantification for the rate of convergence in this equidistribution problem.

\begin{theorem}\label{MainThm-SphCapDisc}
Let $\psi$ be a Hecke eigenfunction of Laplacian eigenvalue $\lambda_{\psi}^2 = \ell_{\psi}(\ell_{\psi}+1)$ for some integer $\ell_{\psi} \geq 0$, normalized such that $\langle \psi, \psi \rangle = 1$. Assume GLH. Then, for all $\varepsilon>0$, we have $$ \mathcal{D}_{\mathrm{cap}}\left( \sigma_{\psi}, \frac{1}{\sigma(\mathcal{O}^{\times} \backslash S^2)} \sigma \right) \ll_{\varepsilon} \ell_{\psi}^{-\frac{1}{2} + \varepsilon}. $$
\end{theorem}

The analogous result to Theorem \ref{MainThm-SphCapDisc} was originally conjectured by Luo and Sarnak in \cite[p.~210]{LuoSarnakQUE}, where it is also mentioned that if true, it would be the best possible result. This conjecture was implicitly proved by Young in \cite{YoungQUEShrinking} assuming GLH, as pointed out by Humphries and Thorner in \cite[p.~2]{HumphriesThorner}.

Implicit in Question \ref{QUEQuestion} is the requirement that equidistribution occurs for shrinking spherical caps $B_R(w)$ centered at $w$ for every $w\in \mathcal{O}^{\times} \backslash S^2$. We can study a related question in which we relax this requirement and ask for equidistribution to occur in shrinking spherical caps $B_R(w)$ centered at $w$ for \textit{almost} every $w\in \mathcal{O}^{\times} \backslash S^2$. The analogous question for the modular surface was answered by Humphries in \cite[Thm.~1.17]{HumphriesPhD}, where he showed that equidistribution holds on almost every shrinking ball when $g$ is a Hecke--Maass form of level $1$ of Laplacian eigenvalue $\frac{1}{4} + t_{g}^2$ and $R\asymp t_{g}^{-\delta}$ for any $0<\delta <1$ (i.e.~at all scales above the Planck scale), conditionally on GLH. We are able to obtain the same result in our context.

\begin{theorem}\label{MainThm-AlmostAllCenters}
Let $\psi$ be a Hecke eigenfunction of Laplacian eigenvalue $\lambda_{\psi}^2 = \ell_{\psi}(\ell_{\psi}+1)$ for some integer $\ell_{\psi}\geq 0$, normalized such that $\langle \psi, \psi\rangle = 1$. Assume GLH and suppose that $R \gg \ell_{\psi}^{-\delta}$ for some $0<\delta < 1$. Then for any $c \gg_{\varepsilon} \ell_{\psi}^{-\frac{1-\delta}{2} + \varepsilon}$, $$ \mathrm{vol}\left( \left\{ w\in \mathcal{O}^{\times} \backslash S^2 : \left| \frac{1}{\mathrm{vol}(B_R)} \int_{B_R(w)} |\psi(z)|^2 \: \mathrm{d}\sigma(z) - \frac{1}{\mathrm{vol}(\mathcal{O}^{\times} \backslash S^2)} \right| > c \right\} \right) $$ converges to zero as $\ell_{\psi}$ tends to infinity.
\end{theorem}

By using Chebyshev's inequality, we have that $$ \mathrm{vol}\left( \left\{ w\in \mathcal{O}^{\times} \backslash S^2 : \left| \frac{1}{\mathrm{vol}(B_R)} \int_{B_R(w)} |\psi(z)|^2 \: \mathrm{d}\sigma(z) - \frac{1}{\mathrm{vol}(\mathcal{O}^{\times} \backslash S^2)} \right| > c \right\} \right) \leq \frac{1}{c^2} \mathrm{Var}(\psi; R),$$ where $$ \mathrm{Var}(\psi; R) := \int_{\mathcal{O}^{\times} \backslash S^2} \left( \frac{1}{\mathrm{vol}(B_R)} \int_{B_R(w)} |\psi(z)|^2 \: \mathrm{d}\sigma(z) - \frac{1}{\mathrm{vol}(\mathcal{O}^{\times} \backslash S^2)} \right)^2 \mathrm{d}\sigma(w).$$ This is the quantity we proceed to bound. Note that, by comparing with \eqref{QUESphericalCapQuestion}, obtaining bounds (or even an asymptotic formula) for the variance $\mathrm{Var}(\psi;R)$ can be interpreted as the $L^2$-version of the shrinking target problem.

The structure of this paper is as follows. In Section \ref{Preliminaries}, we introduce all the necessary background needed to prove the main results, such as the connection between spherical harmonics and automorphic representations of $B^{\times}(\mathbb{A})$, the Watson--Ichino triple product formula, the Selberg--Harish-Chandra transform, and the Berry--Esseen inequality in our context. Then, in Section \ref{Proofs} we prove Theorems \ref{MainThm-Wass}, \ref{MainThm-AllCenters}, \ref{MainThm-SphCapDisc}, and \ref{MainThm-AlmostAllCenters}.

\section{Preliminaries}\label{Preliminaries}

\subsection{Spherical harmonics and automorphic representations of \texorpdfstring{$B^{\times}(\mathbb{A})$}{B(A)}} \label{SphHarmBackground}

The space $L^2(S^2)$ decomposes as $$ L^2(S^2) = \bigoplus_{\ell=0}^{\infty} H_\ell,$$ where $H_\ell$ is the space of harmonic homogeneous polynomials in three variables of degree $\ell$ restricted to $S^2$. In fact, $H_\ell$ is the eigenspace of the eigenvalue $\ell(\ell+1)$ of the Laplace operator $\Delta$. Focusing on the functions fixed by $\mathcal{O}^{\times}$, we have $$ L^2(\mathcal{O}^{\times} \backslash S^2) = \bigoplus_{\ell=0}^{\infty} H_{\ell}^{\mathcal{O}^{\times}}.$$ Thus, we have the following.

\begin{lemma}\label{SpectralRes}
Let $$ f_0(z) := \frac{1}{\sqrt{\mathrm{vol}(\mathcal{O}^{\times} \backslash S^2)}},$$ so that $\langle f_0, f_0 \rangle = 1$, and let $\mathcal{B}$ be an orthonormal basis of spherical harmonics of $L^2(\mathcal{O}^{\times} \backslash S^2)$. Then a function $g \in L^2(\mathcal{O}^{\times} \backslash S^2)$ has the spectral expansion, valid in the $L^2$-sense, of the form $$ g(z) = \langle g, f_0 \rangle f_0(z) + \sum_{f\in \mathcal{B}} \langle g,f\rangle f(z).$$ Moreover, Parseval's identity holds: for $g_1, g_2 \in L^2(\mathcal{O}^{\times} \backslash S^2)$, we have $$ \langle g_1, g_2 \rangle = \langle g_1, f_0\rangle \langle f_0, g_2\rangle + \sum_{f\in \mathcal{B}} \langle g_1, f\rangle \langle f, g_2\rangle.$$
\end{lemma}

Since $H_{\ell}^{\mathcal{O}^{\times}} \subset H_{\ell}$, we have that\footnote{In fact, the Jacquet--Langlands correspondence and the formulas for the dimensions of vector spaces of newforms (see \cite[Thm.~1]{MartinDimensionsNew}) imply that $\dim{H_{\ell}^{\mathcal{O}^{\times}}} \asymp \ell$, but for our purposes we only need an upper bound.} $\dim{H_{\ell}^{\mathcal{O}^{\times}}} \ll \ell$. Hence, we obtain the following estimate which is known as the \textit{Weyl law}.

\begin{proposition}[Weyl law]\label{WeylLaw}
For all $T,t\geq 1$, we have $$ \# \{ \phi \in \mathcal{B} : T \leq \ell_{\phi} \leq T + t \} \ll Tt + t^2. $$
\end{proposition}

There is a localized version of the Weyl law, which gives estimates, on average, of the size of each $\phi \in \mathcal{B}$ at a fixed point $w \in \mathcal{O}^{\times} \backslash S^2$. The estimate is the following.

\begin{proposition}[Local Weyl law {\cite[\S 5, Corollary]{MinakshisundaramPleijelLocalWeylLaw}}]\label{LocalWeylLaw}
For all $T,t\geq 1$ and $w\in \mathcal{O}^{\times} \backslash S^2$, we have $$ \sum_{\substack{\phi \in \mathcal{B} \\ T \leq \ell_{\phi} \leq T+t}} |\phi(w)|^2 \ll Tt + t^2. $$
\end{proposition}

Similar to how there is a connection between modular forms and automorphic representations of $\mathrm{GL}_{2}$, there is a connection between spherical harmonics and automorphic representations of $B^{\times}$. This connection was mentioned in \cite[\S 2.1]{garza2025geodesicrestrictionproblemarithmetic}, but we reproduce it here for the sake of completeness.

For a global field $K$, let $\mathbb{A}_{K}$ be the ring of $K$-adèles and $\mathbb{A}_{K,\text{fin}}$ the ring of finite $K$-adèles; we denote by $\mathbb{A} := \mathbb{A}_{\mathbb{Q}}$ and $\mathbb{A}_{\text{fin}} := \mathbb{A}_{\mathbb{Q},\text{fin}}$. Since $\mathcal{O}$ has class number $1$, we have by \cite[eq.~(2.20)]{AutRepGetz} the homeomorphism 
\begin{align}
    L : \mathcal{O}^{\times} \backslash B^{\times}(\mathbb{R}) &\longrightarrow B^{\times}(\mathbb{Q}) \backslash B^{\times}(\mathbb{A}) / \widehat{\mathcal{O}}^{\times} \label{homeo2.20Bcross} \\
    \mathcal{O}^{\times} x_{\infty} &\longmapsto B^{\times}(\mathbb{Q}) (x_{\infty},1,1,1,\dots ) \widehat{\mathcal{O}}^{\times}, \nonumber
\end{align}
where $\widehat{\mathcal{O}}$ is the closure of $\mathcal{O}$ in $B(\mathbb{A}_{\text{fin}})$. 

Let $\psi \in H_{\ell}^{\mathcal{O}^{\times}}$. We can lift $\psi$ to an eigenfunction $\widehat{\psi}$ of the Casimir operator\footnote{It is a distinguished element of the center of the universal enveloping algebra of the (complex) Lie algebra associated to $B^{\times}(\mathbb{R})$. Intuitively, the Casimir operator is an analogue of the Laplacian operator for Riemannian manifolds in the context of Lie algebras. For more on this operator, see \cite[\S 10.2]{HallLieStuff}.} on $\mathcal{O}^{\times} \backslash B^{\times}(\mathbb{R})$ via the formula $$ \widehat{\psi}(x) := \psi(x\cdot k),$$ where $k$ is the north pole of $S^2$. Hence, we can lift $\widehat{\psi}$ to an automorphic form on $B^{\times}(\mathbb{A})$ given by $L_{*} \widehat{\psi} = \widehat{\psi} \circ L^{-1}$. If $\psi$ is a Hecke eigenfunction, then $L_{*} \widehat{\psi}$ generates an irreducible cuspidal automorphic representation of $B^{\times}(\mathbb{A})$ of trivial central character, denoted by $\rho_{\psi}$. Moreover, $L_{*} \widehat{\psi}$ is the unique $K^{\infty} \widehat{\mathcal{O}}^{\times}$-invariant vector in $\rho_{\psi}$. Additionally, Flath's tensor product theorem (see \cite[Thm.~6.3.4]{AutRepGetz}) gives the tensor product decomposition 
\begin{equation}
    \rho_{\psi} = \bigotimes_{p\leq \infty} {\vphantom{\sum}}' \rho_{p}. \label{Flath}
\end{equation}
The local representations $\rho_p$ are ramified only at $2$ and $\infty$. The representation $\rho_{\infty}$ has as a model the $\ell$-th irreducible representation of $\mathbb{R}^{\times} \backslash B^{\times}(\mathbb{R}) \cong \mathrm{SO}_3(\mathbb{R})$, which is a representation on the space $H_\ell$ (see \cite[Thm.~17.12]{QuantumTheoryforMath}).

\subsection{Watson--Ichino triple product formula}

In the process of proving our theorem, we will need to deal with terms of the form $|\langle |\psi|^2, \phi\rangle |^2$, where $\phi:\mathcal{O}^{\times} \backslash S^2$ is a Hecke eigenfunction. For this, we have the Watson--Ichino formula, originally proved by Watson in a special case and later generalized by Ichino. A priori, this formula is an adèlic statement: the integral of the product of three automorphic forms of $B^{\times}(\mathbb{A})$ over the quotient $\mathbb{A}^{\times} B^{\times}(\mathbb{Q}) \backslash B^{\times}(\mathbb{A})$ is equal to the central value of a triple product $L$-functions involving the Jacquet--Langlands transfers of these automorphic forms. We first mention the adèlic statement, and then present its translation to our context. We follow \cite[\S 4]{HumphriesKhan}.

\subsubsection{Adèlic Watson--Ichino formula}

For $i=1,2,3$, let $\pi_{i}^{B}$ be a (cuspidal) automorphic representation of $B^{\times}(\mathbb{A})$, $\widetilde{\pi}_{i}^{B}$ be the contragredient representation of $\pi_{i}^{B}$, $\pi_{i}$ be the Jacquet--Langlands transfer of $\pi_{i}^{B}$ to a cuspidal automorphic representation of $\mathrm{GL}_2(\mathbb{A})$, and let $\varphi_i\in \pi_{i}^{B}$ and $\widetilde{\varphi}_{i} \in \widetilde{\pi}_{i}^{B}$. Assume that $\varphi_i = \bigotimes_{v} \varphi_{i,v}$ and $\widetilde{\varphi}_{i} = \bigotimes_{v} \widetilde{\varphi}_{i,v}$ are pure tensors in their corresponding automorphic representations under the identification given in Flath's tensor product theorem. Define 
\begin{align*}
    \Phi &:= \varphi_1 \otimes \varphi_2 \otimes \varphi_3, \\
    \widetilde{\Phi} &:= \widetilde{\varphi}_1 \otimes \widetilde{\varphi}_2 \otimes \widetilde{\varphi}_3, \\
    I(\Phi \otimes \widetilde{\Phi}) &:= \left( \int_{Z(\mathbb{A}) B^{\times}(\mathbb{Q}) \backslash B^{\times}(\mathbb{A})} \varphi_1(g) \varphi_2(g) \varphi_3(g) \: \mathrm{d}^{\times}g \right) \left( \int_{Z(\mathbb{A}) B^{\times}(\mathbb{Q}) \backslash B^{\times}(\mathbb{A})} \widetilde{\varphi}_1(g) \widetilde{\varphi}_2(g) \widetilde{\varphi}_3(g) \: \mathrm{d}^{\times}g \right), \\
    \langle \Phi, \widetilde{\Phi} \rangle &:= \prod_{i=1}^{3} \left( \int_{Z(\mathbb{A}) B^{\times}(\mathbb{Q}) \backslash B^{\times}(\mathbb{A})} |\varphi_{i}(g)|^2 \: \mathrm{d}^{\times}g \int_{Z(\mathbb{A}) B^{\times}(\mathbb{Q}) \backslash B^{\times}(\mathbb{A})} |\widetilde{\varphi}_{i}(g)|^2 \: \mathrm{d}^{\times}g \right)^{1/2},
\end{align*}
where $\mathrm{d}^{\times}g$ is the Tamagawa measure, which is normalized so that $Z(\mathbb{A}) B^{\times}(\mathbb{Q}) \backslash B^{\times}(\mathbb{A})$ has volume $2$.\footnote{Notice that the volume of $Z(\mathbb{A}) B^{\times}(\mathbb{Q}) \backslash B^{\times}(\mathbb{A})$ is precisely the Tamagawa number of $B^{\times}$, which is $2$ by \cite[Thm.~3.2.1]{WeilAdelesAlgGroups}.} Similarly, for each place $v$ of $\mathbb{Q}$ with corresponding local field $\mathbb{Q}_{v}$ and maximal compact subgroup $K_v$ of $B^{\times}(\mathbb{Q}_v)$, we let
\begin{align*}
    \Phi_{v} &:= \varphi_{1,v} \otimes \varphi_{2,v} \otimes \varphi_{3,v}, \\
    \widetilde{\Phi}_{v} &:= \widetilde{\varphi}_{1,v} \otimes \widetilde{\varphi}_{2,v} \otimes \widetilde{\varphi}_{3,v}, \\
    \langle \Phi_{v}, \widetilde{\Phi}_{v} \rangle_{v} &:= \prod_{i=1}^{3} \left( \int_{K_v} |\varphi_{i,v}(k_{v})|^2 \: \mathrm{d}k_{v} \int_{K_v} |\widetilde{\varphi}_{i,v}(k_v)|^2 \: \mathrm{d}k_v \right)^{1/2}, \\
    I_v(\Phi_{v} \otimes \widetilde{\Phi}_{v}) &:= \int_{Z(\mathbb{Q}_{v}) \backslash B^{\times}(\mathbb{Q}_{v})} \prod_{i=1}^{3} \langle \pi_{i,v}^{B}(g_v) \varphi_{i,v}, \widetilde{\varphi}_{i,v} \rangle \: \mathrm{d}g_{v}, \\
    I_{v}'(\Phi_{v} \otimes \widetilde{\Phi}_{v} ) &:= \frac{L_v(1,\mathrm{ad}(\pi_{1,v})) L_v(1,\mathrm{ad}(\pi_{2,v})) L_v(1,\mathrm{ad}(\pi_{3,v}))}{\zeta_{v}(2)^2 L_v(\frac{1}{2},\pi_{1,v} \otimes \pi_{2,v} \otimes \pi_{3,v})} \frac{I_v(\Phi_{v} \otimes \widetilde{\Phi}_{v})}{\langle \Phi_{v}, \widetilde{\Phi}_{v} \rangle_v},
\end{align*}
where $$ \zeta_v(s) := \left\{ \begin{array}{ll}
    \frac{1}{1-p^{-s}} & \text{if } v \text{ is nonarchimedean and corresponds to the prime } p, \\
    \pi^{-s/2} \Gamma(s/2) & \text{if } v \text{ is archimedean}
\end{array} \right. $$ is the local factor of the Riemann zeta function $\zeta(s)$ at the place $v$, $\mathrm{d}k_{v}$ are Haar measures, and $\mathrm{d}g_{v}$ are taken in such a way that, if $\mathrm{d}^{\times} g'$ is the Tamagawa measure in $Z(\mathbb{A}) \backslash B^{\times}(\mathbb{A})$, then $$ \mathrm{d}^{\times} g' = \prod_{v} \mathrm{d}g_{v}.$$ The measures $\mathrm{d}g_{v}$ and $\mathrm{d}k_v$ are normalized as follows:
\begin{itemize}
    \item For $v$ nonarchimedean and $v\neq 2$, there exists\footnote{This follows from the fact that all maximal compact subgroups of $\mathrm{GL}_2(\mathbb{Q}_p)$ are conjugate to each other (see \cite[Thm.~7.2.16, p.~327]{BroConOes}).} an isomorphism $\theta_{v} : B^{\times}(\mathbb{Q}_v) \rightarrow \mathrm{GL}_2(\mathbb{Q}_v)$ such that $\theta_v(Z(\mathbb{Q}_v))$ is the group of constant multiples of the identity matrix and $\theta_{v}(K_v) = \mathrm{GL}_2(\mathcal{O}_{v})$. By the Iwasawa decomposition, for $g_{v}' \in \theta_{v}(Z(\mathbb{Q}_v))\backslash \mathrm{GL}_{2}(\mathbb{Q}_v)$ we can write $g_{v}' = \begin{pmatrix}
        a_v & x_v \\ 0 & 1
    \end{pmatrix} k_v$ with $x_v \in \mathbb{Q}_v$, $a_v \in \mathbb{Q}_{v}^{\times}$, and $k_v \in \mathrm{GL}_2(\mathcal{O}_{v})$, where $\mathcal{O}_v$ is the valuation ring of $\mathbb{Q}_v$. Then $$\mathrm{d}g_{v}' = \mathrm{d}x_v \: \mathrm{d}^{\times} a_v \: \mathrm{d}k_{v}',$$ where $\mathrm{d}x_{v}$ is the normalized Haar measure on $\mathbb{Q}_{v}$ such that $\mathcal{O}_{v}$ has measure $1$, $\mathrm{d}^{\times} a_{v} = \zeta_{v}(1) |a_{v}|_{v}^{-1} \: \mathrm{d}a_{v}$ is the multiplicative Haar measure on $\mathbb{Q}_{v}^{\times}$ normalized so that $\mathcal{O}_{v}^{\times}$ has measure $1$, and $\mathrm{d}k_{v}'$ is the Haar probability measure on the compact group $\mathrm{GL}_{2}(\mathcal{O}_{v})$. Finally, the measures $\mathrm{d}g_{v}$ and $\mathrm{d}k_{v}$ on $B^{\times}(\mathbb{Q}_{v})$ and $K_{v}$ are the pullbacks of $\mathrm{d}g_{v}'$ and $\mathrm{d}k_{v}'$ via $\theta_{v}$, respectively.
    \item For $v=2$, we take the measure in $\mathrm{d}g_2$ to be the Haar measure normalized as follows. Let $\mathcal{O}_2$ be the closure of $\mathcal{O}$ in $B(\mathbb{Q}_2)$. By \cite[Exercise 13.1]{Voight}, $\mathcal{O}_2$ is the valuation ring of $B(\mathbb{Q}_2)$, so that $K_2 = \mathcal{O}_{2}^{\times}$ is the maximal compact subgroup of $B^{\times}(\mathbb{Q}_2)$. We normalize $\mathrm{d}g_2$ in such a way that $\mathcal{O}_{2}^{\times}$ modulo the center $Z(\mathbb{Q}_2) = \mathbb{Q}_2^{\times}$ has volume $2$. Additionally, we normalize $\mathrm{d}k_{2}$ so that $K_{2} = \mathcal{O}_{2}^{\times}$ has volume $2$.
    \item For $v=\infty$, we identify $\mathbb{R}^{\times} \backslash B^{\times}(\mathbb{R})$ with the compact group $\mathrm{SO}_{3}(\mathbb{R})$ and take $\mathrm{d}g_{\infty}$ to be the Haar probability measure.
\end{itemize}
Let $\Lambda(s) := \prod_{v} \zeta_v(s)$ denote the completed Riemann zeta function. We then have the following.

\begin{theorem}[{\cite[Thm.~1.1]{IchinoFormula}}]\label{WatsonIchinoAdelic}
The period integral $I(\Phi \otimes \widetilde{\Phi}) / \langle \Phi, \widetilde{\Phi}\rangle$ is equal to $$ \frac{\Lambda(2)^2}{8} \left( \frac{q(\pi_1 \otimes \pi_2 \otimes \pi_3)^{1/2}}{q(\mathrm{ad}(\pi_1)) q(\mathrm{ad}(\pi_2)) q(\mathrm{ad}(\pi_3))} \right)^{-1/2} \frac{\Lambda(\frac{1}{2}, \pi_1 \otimes \pi_2 \otimes \pi_3)}{\Lambda(1, \mathrm{ad}(\pi_1)) \Lambda(1, \mathrm{ad}(\pi_2)) \Lambda(1, \mathrm{ad}(\pi_3))} \prod_{v} I_{v}'(\Phi_{v} \otimes \widetilde{\Phi}_{v}),$$ with $I_{v}'(\Phi_{v} \otimes \widetilde{\Phi}_{v})$ equal to $1$ whenever $\varphi_{1,v}, \varphi_{2,v}, \varphi_{3,v}$ and $\widetilde{\varphi}_{1,v}, \widetilde{\varphi}_{2,v}, \widetilde{\varphi}_{3,v}$ are spherical vectors at the nonarchimedean place $v$.
\end{theorem}

The quantity $I_{v}'(\Phi_{v} \otimes \widetilde{\Phi}_v)$ is often called the \textit{local constant}. We note that the spherical vectors exist only when $v\notin\{\infty,2\}$ and, in fact, for our choice of vectors (which we specify below), the local vectors at $v\notin \{\infty,2\}$ are precisely the spherical vectors. Thus, we only need to compute such local constants at $2$ and at $\infty$ with a specific combination of local vectors. We now specify our choice.

We take $\pi_{1}^{B} = \pi_{2}^{B} = \rho_{\psi}$ and $\pi_{3}^{B} = \rho_{\phi}$, where $\psi \in H_{\ell_{\psi}}^{\mathcal{O}^{\times}}$ and $\psi \in H_{\ell_{\phi}}^{\mathcal{O}^{\times}}$. Note that these automorphic representations are self-dual since the Hecke eigenvalues are real, so that their contragredient representations are the same ones. The vectors we choose are $\varphi_1 = \varphi_2 = \widetilde{\varphi}_1 = \widetilde{\varphi}_2 = L_{*}\widehat{\psi} / ||L_{*}\widehat{\psi}||_{L^2}$ and $\varphi_3 = \widetilde{\varphi}_3 = L_{*}\widehat{\phi}/ ||L_{*}\widehat{\phi}||_{L^2}$, so that they are $L^2$-normalized. We denote by $\pi_{\psi}$ and $\pi_{\phi}$ the Jacquet--Langlands transfers of $\rho_{\psi}$ and $\rho_{\phi}$ to (cuspidal) automorphic representations of $\mathrm{GL}_2(\mathbb{A})$, respectively.

\subsubsection{Computation of archimedean local constant}

We start with the archimedean place. As mentioned in Section \ref{SphHarmBackground}, the local representation $\rho_{\psi,\infty}$ has as a model the $\ell_{\psi}$-th irreducible representation of $\mathbb{R}^{\times} \backslash B^{\times}(\mathbb{R}) \cong \mathrm{SO}_{3}(\mathbb{R})$, which is a representation on the space $H_{\ell_{\psi}}$. Since the local vector we consider is the projection of $L_{*}\widehat{\psi} / ||L_{*}\widehat{\phi}||_{L^2}$ to the local vector at $\infty$ and $L_{*}\widehat{\psi}/ ||L_{*}\widehat{\phi}||_{L^2}$ is $L^2$-normalized and $K^{\infty}$-invariant, then the local vectors must also be $L^2$-normalized and $K^{\infty}$-invariant. There is a unique vector with this property: define the $\ell$-th Legendre polynomial by $$P_{\ell}(t) := \frac{1}{2^{\ell} \ell!} \frac{\mathrm{d}^{\ell}}{\mathrm{d}t^{\ell}}(t^2-1)^{\ell}$$ and let $Y_{\ell_{\psi}}$ be the Legendre polynomial $P_{\ell_{\psi}}$ normalized so that $$||Y_{\ell_{\psi}}(\langle \cdot, k\rangle )||_{L^2(S^2)}=1;$$ then $Y_{\ell_{\phi}}(\langle \cdot, k\rangle )$ is $K^{\infty}$-invariant and is the unique $L^2$-normalized vector in $H_{\ell_{\psi}}$ with this property. Hence, the local vector at $\infty$ corresponding to $L_{*}\widehat{\psi}$ is precisely $Y_{\ell_{\psi}}(\langle \cdot, k\rangle )$. Similarly, if $\phi$ is of degree $\ell_{\phi}$, then the local vector at $\infty$ corresponding to $L_{*}\widehat{\phi}$ is $Y_{\ell_{\phi}}(\langle \cdot, k\rangle )$. With this in mind, we compute the local constant at $v=\infty$.

\begin{lemma}\label{LocalConstInfty}
With the above choices, if $\ell_{\phi}$ is even and $0\leq \ell_{\phi} \leq 2\ell_{\psi}$, we have
\begin{equation*}
    I_{\infty}'(\Phi_{\infty} \otimes \widetilde{\Phi}_{\infty}) = \frac{\sqrt{\pi} L_{\infty}(1,\mathrm{ad}(\pi_{\psi}))^2 L_{\infty}(1,\mathrm{ad}(\pi_{\phi}))}{8 L_{\infty}(\frac{1}{2},\pi_{\psi}\otimes \pi_{\psi} \otimes \pi_{\phi})} \times \frac{(\ell_{\phi}/2+1)^2 (\ell_{\psi} - \ell_{\phi}/2 + 1)C_{\ell_{\phi}/2}^{2} C_{\ell_{\psi} - \ell_{\phi}/2}}{(2\ell_{\psi} + \ell_{\phi} + 1) (\ell_{\psi} + \ell_{\phi}/2 + 1) C_{\ell_{\psi} + \ell_{\phi}/2}},
\end{equation*}
where $C_m := \frac{1}{m+1}\binom{2m}{m}$ is the $m$-th Catalan number. Otherwise, $I_{\infty}'(\Phi_{\infty} \otimes \widetilde{\Phi}_{\infty}) = 0$.
\end{lemma}

\begin{proof}
Since $\zeta_{\infty}(2) = \frac{1}{\pi}$ and $\langle \Phi_{\infty}, \widetilde{\Phi}_{\infty} \rangle_{\infty} = 1$, we only need to compute $I_{\infty}(\Phi_{\infty} \otimes \widetilde{\Phi}_{\infty})$. By definition, we have 
\begin{align}
    &I_{\infty}(\Phi_{\infty} \otimes \widetilde{\Phi}_{\infty}) \label{LocalConstInfty-eq1} \\
    &= \int_{\mathbb{R}^{\times} \backslash B^{\times}(\mathbb{R})} \left[ \left( \int_{S^2} Y_{\ell_{\psi}}(\langle g\cdot x, k \rangle ) Y_{\ell_{\psi}}(\langle x, k\rangle ) \: \mathrm{d}\sigma(x) \right)^2 \left( \int_{S^2} Y_{\ell_{\phi}}(\langle g\cdot x, k \rangle ) Y_{\ell_{\phi}}(\langle x, k\rangle ) \: \mathrm{d}\sigma(x) \right) \right] \: \mathrm{d}g. \nonumber
\end{align}
Since $\sqrt{2\ell_{\psi}+1}Y_{\ell_{\psi}}(\langle \cdot, \cdot \rangle)$ is the reproducing kernel of $H_{\ell_{\psi}}$ in $L^2(S^2)$ (see \cite[\S 2.3, Lemma 2.22]{MorimotoSphere}) and $B^{\times}(\mathbb{R})$ acts on $S^2$ by isometries, the first inner integral on the right-hand side of \eqref{LocalConstInfty-eq1} is equal to $$ \int_{S^2} Y_{\ell_{\psi}}(\langle g^{-1}\cdot k, x\rangle) Y_{\ell_{\psi}}(\langle x,k\rangle) \: \mathrm{d}\sigma(x) = \frac{1}{\sqrt{2\ell_{\psi}+1}} Y_{\ell_{\psi}}(\langle g^{-1}\cdot k, k\rangle ).$$ Similarly, the second inner integral on the right-hand side of \eqref{LocalConstInfty-eq1} is equal to $\frac{1}{\sqrt{2\ell_{\phi}+1}} Y_{\ell_{\phi}}(\langle g^{-1}\cdot k, k\rangle )$. Hence, $$ I_{\infty}(\Phi_{\infty} \otimes \widetilde{\Phi}_{\infty}) = \frac{1}{(2\ell_{\psi}+1)\sqrt{2\ell_{\phi}+1}} \int_{\mathbb{R}^{\times} \backslash B^{\times}(\mathbb{R})} Y_{\ell_{\psi}}(\langle g^{-1} \cdot k, k\rangle )^2 \: Y_{\ell_{\phi}}(\langle g^{-1}\cdot k, k\rangle) \: \mathrm{d}g.$$ Recall that $\mathbb{R}^{\times} \backslash B^{\times}(\mathbb{R}) \cong \mathrm{SO}_{3}(\mathbb{R})$. Moreover,
\begin{align*}
    \mathrm{SO}_{3}(\mathbb{R}) / K^{\infty} &\longrightarrow S^2 \\
    gK^{\infty} &\longmapsto g\cdot k
\end{align*}
is a homeomorphism, and $K^{\infty} \cong S^2$, which possesses a Haar probability measure. From these facts and using \cite[Thm.~2.51]{FollandHarmonic}, we obtain that
\begin{align*}
    I_{\infty}(\Phi_{\infty} \otimes \widetilde{\Phi}_{\infty}) &= \frac{1}{(2\ell_{\psi}+1)\sqrt{2\ell_{\phi}+1}} \int_{\mathrm{SO}_3(\mathbb{R})} Y_{\ell_{\psi}}(\langle k, g\cdot k\rangle )^2 \: Y_{\ell_{\phi}}(\langle k, g\cdot k\rangle) \: \mathrm{d}g \\
    &= \frac{1}{(2\ell_{\psi}+1)\sqrt{2\ell_{\phi}+1}} \int_{\mathrm{SO}_3(\mathbb{R})/K^{\infty}} Y_{\ell_{\psi}}(\langle k, g\cdot k\rangle )^2 \: Y_{\ell_{\phi}}(\langle k, g\cdot k\rangle) \: \mathrm{d}g \\
    &= \frac{1}{4\pi (2\ell_{\psi}+1)\sqrt{2\ell_{\phi}+1}} \int_{S^2} Y_{\ell_{\psi}}(\langle k, x\rangle )^2 \: Y_{\ell_{\phi}}(\langle k, x\rangle) \: \mathrm{d}\sigma(x) \\
    &= \frac{1}{4\pi (2\ell_{\psi}+1)\sqrt{2\ell_{\phi}+1}} \int_{S^2} Y_{\ell_{\psi}}(\langle x, k\rangle)^2 \: Y_{\ell_{\phi}}(\langle x, k\rangle ) \: \mathrm{d}\sigma(x).
\end{align*}
By the Gaunt integral relation (see \cite[Appendix A]{CruzanProdsSpherHarms}), we obtain that $I_{\infty}(\Phi_{\infty} \otimes \widetilde{\Phi}_{\infty}) = 0$ unless $\ell_{\phi}$ is even and $0 \leq \ell_{\phi} \leq 2\ell_{\psi}$, in which case we have
\begin{align*}
    I_{\infty}(\Phi_{\infty}\otimes \widetilde{\Phi}_{\infty}) &= \frac{1}{8\pi^{3/2}} \times \frac{(\ell_{\phi}!)^2 (2\ell_{\psi}-\ell_{\phi})!}{(2\ell_{\psi} + \ell_{\phi}+1)!} \times \frac{((\ell_{\psi} + \ell_{\phi}/2)!)^2}{((\ell_{\phi}/2)!)^4 ((\ell_{\psi} - \ell_{\phi}/2)!)^2} \\
    &= \frac{(\ell_{\phi}/2+1)^2 (\ell_{\psi} - \ell_{\phi}/2 + 1)C_{\ell_{\phi}/2}^{2} C_{\ell_{\psi} - \ell_{\phi}/2}}{8\pi^{3/2} (2\ell_{\psi} + \ell_{\phi} + 1) (\ell_{\psi} + \ell_{\phi}/2 + 1) C_{\ell_{\psi} + \ell_{\phi}/2}}. \qedhere
\end{align*}
\end{proof}

Using Stirling's estimate of the Gamma function (see \cite[Thm.~C.1]{MontgomeryVaughan}) in the form $$n! \asymp n^{n+\frac{1}{2}} e^{-n},$$ we obtain the following estimate for the local constant at $\infty$.

\begin{corollary}\label{LocalConstInftyEstimate}
With the above choices, if $\ell_{\phi}$ is even and $0\leq \ell_{\phi} \leq 2\ell_{\psi}$, we have $$ I_{\infty}'(\Phi_{\infty} \otimes \widetilde{\Phi}_{\infty}) \asymp \frac{L_{\infty}(1,\mathrm{ad}(\pi_{\psi}))^2 L_{\infty}(1,\mathrm{ad}(\pi_{\phi}))}{L_{\infty}(\frac{1}{2},\pi_{\psi}\otimes \pi_{\psi} \otimes \pi_{\phi})} \times \frac{1}{(\ell_{\phi}+1) (1 + 2\ell_{\psi} - \ell_{\phi})^{1/2}(2\ell_{\psi} + \ell_{\phi})^{1/2}}.$$ Otherwise, $I_{\infty}'(\Phi_{\infty} \otimes \widetilde{\Phi}_{\infty}) = 0$.
\end{corollary}

\subsubsection{Computation of nonarchimedean local constant}

We now compute the local constant at $v=2$. We follow \cite[Prop. 4.5]{WoodburyLocalConst} and \cite[Thm. 2]{WatsonPhD}. We have that $B(\mathbb{Q}_2)$ is a central division algebra over $\mathbb{Q}_2$ (since $B$ is ramified at $2$). By \cite[Thm.~13.1.6, Lemma 13.3.2]{Voight}, if $\alpha=i+j \in B(\mathbb{Q}_2)$, then $\alpha^2=-2$ and $B^{\times}(\mathbb{Q}_{2}) = \{ \alpha^n : n\in \mathbb{Z} \} \times \mathcal{O}_{2}^{\times}$. Since $\varphi_{i}$ is $\mathbb{A}^{\times}$-invariant and right $\widehat{\mathcal{O}}^{\times}$-invariant, then $\varphi_{i,2}$ must be $\mathbb{Q}_{2}^{\times}$-invariant and right $\mathcal{O}_{2}^{\times}$-invariant. This implies that $\varphi_{i,2}$ is $\alpha^2$-invariant. Additionally, since $\varphi_{i}$ is $L^2$-normalized, then $\varphi_{i,2}$ is $L^2$-normalized as well, so that we may assume, without loss of generality, that $\varphi_{i,2}$ is equal to $1/\sqrt{2}$ on $\mathcal{O}_{2}^{\times}$ (due to the normalization of $\mathrm{d}k_2$). Hence, $\varphi_{i,2}$ is completely determined by its value on $\alpha$, which will be $\pm 1/\sqrt{2}$ by the above analysis. For $i=1,2,3$, let $\alpha_{i} := \sqrt{2} \varphi_{i,2}(\alpha)\in \{\pm 1\}$ and set $\varepsilon := \alpha_1 \alpha_2 \alpha_3 \in \{\pm 1\}$.

\begin{lemma}\label{LocalConst2}
Under the above choices, we have
$$ I_{2}'(\Phi_2 \otimes \widetilde{\Phi}_2) = (1+\varepsilon)\frac{9L_2(1,\mathrm{ad}(\pi_{\psi,2}))^2 L_2(1,\mathrm{ad}(\pi_{\phi,2}))}{8 L_2(\frac{1}{2}, \pi_{\psi,2} \otimes \pi_{\psi,2}\otimes \pi_{\phi,2})}. $$
\end{lemma}

\begin{proof}
Since $\zeta_2(2) = (1-2^{-2})^{-1} = 4/3$ and $\langle \Phi_2, \widetilde{\Phi}_2\rangle_2 = 1$, it suffices to compute $I_2(\Phi_2\otimes \widetilde{\Phi}_2)$. For this, since $\varphi_1, \varphi_2, \varphi_3$ are right $\mathcal{O}_{2}^{\times}$-invariant and the measure of $\mathcal{O}_{2}^{\times}$ modulo the center is $2$, we have that $$ I_2(\Phi_2 \otimes \widetilde{\Phi}_2) = 2\int_{\mathbb{Q}_{2}^{\times} \backslash B^{\times}(\mathbb{Q}_2) / \mathcal{O}_{2}^{\times}} \prod_{i=1}^{3} \langle \pi_{i,2}^{B}(g_2) \varphi_{i,2}, \widetilde{\varphi}_{i,2}\rangle_{2} \: \mathrm{d}g_2.$$ The set $\mathbb{Q}_{2}^{\times} \backslash B^{\times}(\mathbb{Q}_2) / \mathcal{O}_{2}^{\times}$ has only two elements: $\mathbb{Q}_{2}^{\times} \mathcal{O}_{2}^{\times}$ and $\mathbb{Q}_{2}^{\times} \alpha \mathcal{O}_{2}^{\times}$. It follows that $$ I_2(\Phi_2 \otimes \widetilde{\Phi}_2) = 2\prod_{i=1}^{3} \langle \pi_{i,2}^{B}(1)\varphi_{i,2}, \widetilde{\varphi}_{i,2}\rangle_{2} + 2\prod_{i=1}^{3} \langle \pi_{i,2}^{B}(\alpha)\varphi_{i,2}, \widetilde{\varphi}_{i,2}\rangle_{2}.$$ Using the above analysis, we obtain $$ I_2(\Phi_2 \otimes \widetilde{\Phi}_2) = 2\prod_{i=1}^{3} \langle \varphi_{i,2}, \widetilde{\varphi}_{i,2}\rangle_{2} + 2\prod_{i=1}^{3} \langle \varphi_{i,2}(\cdot \: \alpha), \widetilde{\varphi}_{i,2}\rangle_{2} = 2 + 2\varepsilon,$$ from where the result follows.
\end{proof}

We proceed to estimate the local $L$-factors in Lemma \ref{LocalConst2}. The local $L$-factors of the adjoint $L$-function of $\pi_{\psi}$ and $\pi_{\phi}$ at the place $2$ were computed in \cite[\S 4.7]{garza2025geodesicrestrictionproblemarithmetic}, from where we know that $$L_2(1,\mathrm{ad}(\pi_{\psi,2})) = L_2(1,\mathrm{ad}(\pi_{\phi,2})) = 4/3.$$ Hence, we concentrate on computing $L_2(\frac{1}{2}, \pi_{\psi,2} \otimes \pi_{\psi,2} \otimes \pi_{\phi,2})$. For the necessary background on the following discussion, see \cite[\S 7]{localLanglandsGL2}, \cite[\S 1]{IntroLanglands}, and \cite{NTBTate}.

Recall that $\pi_{\psi}$ has conductor $2$ since it corresponds to $f_{\psi}$, as constructed in \eqref{JLModForm}, and $f_{\psi}$ has level $2$. Additionally, $\pi_{\psi}$ has trivial central character since $f_{\psi}$ has trivial character. Then, by \cite[Prop.~2.8]{LocalComponents}, we know that $\pi_{\psi,2}$ is the twist of the Steinberg representation by a (unitary) unramified character $\chi_{\psi}$ of $\mathbb{Q}_{2}^{\times}$, i.e.~$\pi_{\psi,2} \simeq \mathrm{St}\otimes \chi_{\psi}$, with $\chi_{\psi}^{2} = 1$ since the central character of $\mathrm{St}\otimes \chi_{\psi}$ is $\chi_{\psi}^2$. Consequently, $\pi_{\psi,2}$ corresponds under the local Langlands correspondence to the Weil--Deligne parameters $(\tau_{\psi}, V_{\psi}, N_{\psi})$, where $\tau_{\psi}:W_{\mathbb{Q}_2} \rightarrow \mathrm{GL}_2(\mathbb{C})$ is a representation of the Weil group of $\mathbb{Q}_2$ on the $\mathbb{C}$-vector space $V_{\psi}$ such that 
\begin{equation}
    \tau_{\psi}(\Psi) = \begin{pmatrix}
        \chi_{\psi}(\Psi)2^{-1/2} & 0 \\ 0 & \chi_{\psi}(\Psi) 2^{1/2}
    \end{pmatrix}, \label{imageFrob}
\end{equation}
where $\Psi$ is the geometric Frobenius, and 
\begin{equation}
    N_{\psi} = \begin{pmatrix}
        0 & 1 \\ 0 & 0
    \end{pmatrix} \label{NilpotentEndo}
\end{equation}
is a nilpotent endomorphism of $\mathbb{C}^2$. Similarly, $\pi_{\phi,2} \simeq \mathrm{St}\otimes \chi_{\phi}$ where $\chi_{\phi}$ is a (unitary) unramified character of of $\mathbb{Q}_{2}^{\times}$, and $\pi_{\phi,2}$ corresponds under the local Langlands correspondence to the Weil--Deligne parameters $(\tau_{\phi}, V_{\phi}, N_{\phi})$ that satisfy similar properties as described in \eqref{imageFrob} and \eqref{NilpotentEndo}.

It follows that $\pi_{\psi,2} \otimes \pi_{\psi,2} \otimes \pi_{\phi,2}$ corresponds under the local Langlands correspondence to the Weil--Deligne parameters $(\tau, V, N)$, where 
\begin{align*}
    \tau &= \tau_{\psi} \otimes \tau_{\psi} \otimes \tau_{\phi}, \\
    V &= V_{\psi} \otimes V_{\psi} \otimes V_{\phi}, \\
    N &= N_{\psi} \otimes \mathrm{Id}_{V_{\psi}} \otimes \mathrm{Id}_{V_{\phi}} + \mathrm{Id}_{V_{\psi}} \otimes N_{\psi} \otimes \mathrm{Id}_{V_{\phi}} + \mathrm{Id}_{V_{\psi}} \otimes \mathrm{Id}_{V_{\psi}} \otimes N_{\phi}.
\end{align*}
Let $\{v_1, v_2\}$ be a basis of $V_{\psi}$ such that \eqref{imageFrob} and \eqref{NilpotentEndo} hold. Similarly, let $\{w_1, w_2\}$ be a basis of $V_{\phi}$ such that the corresponding equalities hold. Then, a basis for $V_{\psi} \otimes V_{\psi} \otimes V_{\phi}$ is given by $$ \{v_1 \otimes v_1 \otimes w_1, v_1 \otimes v_1 \otimes w_2, v_1 \otimes v_2 \otimes w_1, v_1 \otimes v_2 \otimes w_2, v_2 \otimes v_1 \otimes w_1, v_2 \otimes v_1 \otimes w_2, v_2 \otimes v_2 \otimes w_1, v_2 \otimes v_2 \otimes w_2\}. $$ Under this basis, we have
\begin{align}
    \tau(\Psi) &= \chi_{\psi}(\Psi)^2 \chi_{\phi}(\Psi) \: \mathrm{diag}\left(2^{-3/2}, 2^{-1/2}, 2^{-1/2}, 2^{1/2}, 2^{-1/2}, 2^{1/2}, 2^{1/2}, 2^{3/2}\right) \label{imageFrobTensorProd} \\
    &= \chi_{\phi}(\Psi) \: \mathrm{diag}\left(2^{-3/2}, 2^{-1/2}, 2^{-1/2}, 2^{1/2}, 2^{-1/2}, 2^{1/2}, 2^{1/2}, 2^{3/2}\right), \nonumber
\end{align}
where $\mathrm{diag}(x_1, \dots, x_n)$ is the $n\times n$ diagonal matrix with $x_1, \dots, x_n$ being the entries in the diagonal, and 
\begin{equation}
    N = \begin{pmatrix}
        0 & 1 & 1 & 0 & 1 & 0 & 0 & 0 \\
        0 & 0 & 0 & 1 & 0 & 1 & 0 & 0 \\
        0 & 0 & 0 & 1 & 0 & 0 & 1 & 0 \\
        0 & 0 & 0 & 0 & 0 & 0 & 0 & 1 \\
        0 & 0 & 0 & 0 & 0 & 1 & 1 & 0 \\
        0 & 0 & 0 & 0 & 0 & 0 & 0 & 1 \\
        0 & 0 & 0 & 0 & 0 & 0 & 0 & 1 \\
        0 & 0 & 0 & 0 & 0 & 0 & 0 & 0
    \end{pmatrix}. \label{NilpotentTensorProd}
\end{equation}
Since the last five columns of $N$ are linearly independent and the first three can be obtained from the fifth one, we have that $N$ has rank $5$, so that $\mathrm{ker}(N)$ has dimension $3$. Moreover, from \eqref{NilpotentTensorProd}, we can see that a basis for $\mathrm{ker}(N)$ is given by $$\{v_1\otimes v_1\otimes w_1, v_2 \otimes v_1 \otimes w_1 - v_1\otimes v_1\otimes w_2, v_1 \otimes v_2 \otimes w_1 - v_1\otimes v_1 \otimes w_2\}.$$ By \eqref{imageFrobTensorProd}, we have that $\tau(\Psi)$ acts on the first generator by multiplication by $\chi_{\phi}(\Psi) 2^{-3/2}$, whereas $\tau(\Psi)$ acts on the second and third generators by multiplication by $\chi_{\phi}(\Psi) 2^{-1/2}$. It follows that 
\begin{align*}
    &L_2\left(\frac{1}{2}, \pi_{\psi,2} \otimes \pi_{\psi,2} \otimes \pi_{\phi,2} \right) \\
    &= \mathrm{det}\left( \mathrm{Id}_{\mathrm{ker}(N)} - \tau(\Psi)|_{\mathrm{ker}(N)} 2^{-1/2} \right)^{-1} \\
    &= \frac{1}{\left(1 - \chi_{\phi}(\Psi) 2^{-2} \right) \left(1 - \chi_{\phi}(\Psi) 2^{-1} \right) \left(1 - \chi_{\phi}(\Psi) 2^{-1} \right)},
\end{align*}
which is bounded when $\psi$ and $\phi$ vary since $\chi_{\phi}$ is a unitary character of $\mathbb{Q}_{2}^{\times}$. From the previous discussion, we conclude the following.

\begin{corollary}\label{LocalConst2Estimate}
Under the above choices, we have $$ I_{2}'(\Phi_2 \otimes \widetilde{\Phi}_2) \asymp 1.$$
\end{corollary}

\subsubsection{Classical Watson--Ichino formula}

We restate the adèlic Watson--Ichino formula in our context. Let $\psi\in H_{\ell_{\psi}}^{\mathcal{O}^{\times}}$ and $\phi \in H_{\ell_{\phi}}^{\mathcal{O}^{\times}}$ be Hecke eigenfunctions.

\begin{corollary}\label{WatsonIchinoClassical}
If $\ell_{\phi}$ is even and $0<\ell_{\phi} \leq 2\ell_{\psi}$, we have $$\left| \langle |\psi|^2, \phi \rangle \right|^2 \asymp \frac{1}{(\ell_{\phi}+1) (1+2\ell_{\psi} - \ell_{\phi})^{1/2}(2\ell_{\psi} + \ell_{\phi})^{1/2}} \times \frac{L(\frac{1}{2}, \pi_{\psi} \otimes \pi_{\psi} \otimes \pi_{\phi})}{L(1, \mathrm{ad}(\pi_{\psi}))^2 L(1, \mathrm{ad}(\pi_{\phi}))}.$$ If $\ell_{\phi}=0$, then $\langle |\psi|^2, \phi \rangle = 1$. Otherwise, $\langle |\psi|^2, \phi\rangle = 0$.
\end{corollary}

\begin{proof} 
The case $\ell_{\phi}=0$ is immediate, so we assume $\ell_{\phi}>0$. We let $\varphi_1 = \varphi_2 = \widetilde{\varphi}_1 = \widetilde{\varphi}_2 = L_{*}\widehat{\psi}$ and $\varphi_3 = \widetilde{\varphi}_3 = L_{*}\widehat{\phi}$. By Theorem \ref{WatsonIchinoAdelic} and Corollaries \ref{LocalConstInftyEstimate} and \ref{LocalConst2Estimate}, we have\footnote{Note that $I(\Phi \otimes \widetilde{\Phi})/\langle \Phi \otimes \widetilde{\Phi}\rangle$ is a trilinear form, so scaling the vectors we use for $\varphi_i$ does not change the value of the trilinear form; in particular, it does not change if we consider $\varphi_1 = L_{*}\widehat{\psi}$ or $\varphi_1 = L_{*}\widehat{\psi} / ||L_{*}\widehat{\psi}||_{L^2}$.}
\begin{align}
    \frac{I(\Phi \otimes \widetilde{\Phi})}{\langle \Phi, \widetilde{\Phi} \rangle} &= \frac{\Lambda(2)^2}{8} \left( \frac{q(\pi_1 \otimes \pi_2 \otimes \pi_3)^{1/2}}{q(\mathrm{ad}(\pi_1)) q(\mathrm{ad}(\pi_2)) q(\mathrm{ad}(\pi_3))} \right)^{-1/2} \frac{\Lambda(\frac{1}{2}, \pi_1 \otimes \pi_2 \otimes \pi_3)}{\Lambda(1, \mathrm{ad}(\pi_1)) \Lambda(1, \mathrm{ad}(\pi_2)) \Lambda(1, \mathrm{ad}(\pi_3))} \label{WatsonIchinoClassical-eq1} \\
    &\quad \times I_{\infty}'(\Phi_{\infty} \otimes \widetilde{\Phi}_{\infty}) I_{2}'(\Phi_{2} \otimes \widetilde{\Phi}_{2}) \nonumber \\
    &\asymp \frac{1}{(\ell_{\phi}+1) (1+2\ell_{\psi} - \ell_{\phi})^{1/2}(2\ell_{\psi} + \ell_{\phi})^{1/2}} \times \frac{L(\frac{1}{2}, \pi_{\psi} \otimes \pi_{\psi} \otimes \pi_{\phi})}{L(1, \mathrm{ad}(\pi_{\psi}))^2 L(1, \mathrm{ad}(\pi_{\phi}))} \nonumber
\end{align} 
if $\ell_{\phi} \leq 2\ell_{\psi}$; otherwise, $I(\Phi \otimes \widetilde{\Phi})=0$. Now, for $i=1,2,3$, we know that $\varphi_i$ is invariant under $\widehat{\mathcal{O}}^{\times}$, where $\mathcal{O}$ is the ring of Hurwitz integers in $B(\mathbb{Q})$ and $\widehat{\mathcal{O}}$ is its closure in $B(\mathbb{A}_{\text{fin}})$. The Tamagawa measure is constructed in such a way that the volume of $\widehat{\mathcal{O}}^{\times}$ in $\mathbb{A}^{\times} B^{\times}(\mathbb{Q}) \backslash B^{\times}(\mathbb{A})$ is $1$. It follows that 
\begin{align*}
    \langle \Phi, \widetilde{\Phi} \rangle &= \prod_{i=1}^{3} \left( \int_{\mathbb{A}^{\times} B^{\times}(\mathbb{Q}) \backslash B^{\times}(\mathbb{A}) / \widehat{\mathcal{O}}^{\times}} |\varphi_{i}(g)|^2 \: \mathrm{d}^{\times}g \int_{\mathbb{A}^{\times} B^{\times}(\mathbb{Q}) \backslash B^{\times}(\mathbb{A}) / \widehat{\mathcal{O}}^{\times}} |\widetilde{\varphi}_{i}(g)|^2 \: \mathrm{d}^{\times}g \right)^{1/2} \\
    &= \prod_{i=1}^{3} \int_{\mathbb{A}^{\times} B^{\times}(\mathbb{Q}) \backslash B^{\times}(\mathbb{A}) / \widehat{\mathcal{O}}^{\times}} |\varphi_{i}(g)|^2 \: \mathrm{d}^{\times}g.
\end{align*}
However, by the homeomorphism \eqref{homeo2.20Bcross}, we have that the map
\begin{align}
    \mathcal{O}^{\times} \backslash S^2 &\longrightarrow \mathbb{A}^{\times} B^{\times}(\mathbb{Q}) \backslash B^{\times}(\mathbb{A}) / \widehat{\mathcal{O}}^{\times} \label{HomeoSphereAdeles} \\
    \mathcal{O}^{\times} g_{\infty} &\longmapsto \mathbb{A}^{\times} B^{\times}(\mathbb{Q}) (g_{\infty},1,1,\dots) \widehat{\mathcal{O}}^{\times} \nonumber
\end{align}
is a homeomorphism. Using this homeomorphism, the construction of $\varphi_i$ from $\psi$ and $\phi$, and considering the normalization of volumes, we obtain that 
\begin{equation}
    \langle \Phi, \widetilde{\Phi} \rangle = \left( \frac{1}{2\pi} \int_{S^2} |\psi(z)|^{2} \: \mathrm{d}\sigma(z) \right)^2 \left( \frac{1}{2\pi} \int_{S^2} |\phi(z)|^{2} \: \mathrm{d}\sigma(z) \right) = \frac{1}{8\pi^3} \label{WatsonIchinoClassical-eq2}
\end{equation}
since $\psi$ and $\phi$ are $L^2(S^2)$-normalized. By a similar reasoning and using the homeomorphism \eqref{HomeoSphereAdeles}, since we took Hecke eigenfunctions to be real-valued, we have that 
\begin{align}
    I(\Phi \otimes \widetilde{\Phi}) &= \left( \int_{\mathbb{A}^{\times} B^{\times}(\mathbb{Q}) \backslash B^{\times}(\mathbb{A})} |\varphi_1(g)|^2 \varphi_3(g) \: \mathrm{d}^{\times}g \right)^2 \label{WatsonIchinoClassical-eq3} \\
    &= \left( \int_{\mathbb{A}^{\times} B^{\times}(\mathbb{Q}) \backslash B^{\times}(\mathbb{A}) / \widehat{\mathcal{O}}^{\times}} |\varphi_1(g)|^2 \varphi_3(g) \: \mathrm{d}^{\times}g \right)^2 \nonumber \\
    &= \left( \frac{6}{\pi} \int_{\mathcal{O}^{\times} \backslash S^2} |\psi(z)|^2 \phi(z) \: \mathrm{d}\sigma(z) \right)^2 = \frac{36}{\pi} \left|\langle |\psi|^2, \phi \rangle \right|^2. \nonumber
\end{align}
The result follows from combining \eqref{WatsonIchinoClassical-eq1}, \eqref{WatsonIchinoClassical-eq2}, and \eqref{WatsonIchinoClassical-eq3}.
\end{proof}

\subsection{\texorpdfstring{$L$}{L}-functions}

As mentioned in the previous section, we will work with the $L$-functions $$ L(s,\pi_{\psi} \otimes \pi_{\psi} \otimes \pi_{\phi}) \: \text{ and} \: L(s,\mathrm{ad}(\pi_{\phi})),$$ where $\psi \in H_{\ell_{\psi}}^{\mathcal{O}^{\times}}$, $\phi \in H_{\ell_{\phi}}^{\mathcal{O}^{\times}}$, and $\pi_{\psi}$ and $\pi_{\phi}$ are the Jacquet--Langlands transfers of $\rho_{\psi}$ and $\rho_{\phi}$ to automorphic representations of $\mathrm{GL}_2(\mathbb{A})$, respectively. By Corollary \ref{WatsonIchinoClassical}, we only need to consider the case when $0< \ell_{\phi} \leq 2\ell_{\psi}$ and $\ell_{\phi}$ is even. In this case, we need to compute the Langlands parameters of these $L$-functions. We do this in the following lemma.

\begin{lemma}\label{LfuncsData}
For $\psi \in H_{\ell_{\psi}}^{\mathcal{O}^{\times}}$ and $\phi \in H_{\ell_{\phi}}^{\mathcal{O}^{\times}}$, the Langlands parameters of $\mathrm{ad}(\pi_{\psi})$ are $$\{1,2\ell_{\psi}+1, 2\ell_{\psi}+2\},$$ whereas the Langlands parameters of $\pi_{\psi} \otimes \pi_{\psi} \otimes \pi_{\phi}$ are $$ \left\{\ell_{\phi}+\frac{1}{2}, \ell_{\phi}+\frac{1}{2}, \ell_{\phi}+\frac{3}{2}, \ell_{\phi}+\frac{3}{2}, 2\ell_{\psi} - \ell_{\phi} + \frac{1}{2}, 2\ell_{\psi} - \ell_{\phi} + \frac{3}{2}, 2\ell_{\psi} + \ell_{\phi} + \frac{3}{2}, 2\ell_{\psi} + \ell_{\phi} + \frac{5}{2}\right\}.$$
\end{lemma}

\begin{proof}
The Langlands parameters of $\mathrm{ad}(\pi_{\psi})$ can be found in \cite[\S 5.12]{IwaniecKowalski}, so we focus on the triple product. Since $f_{\psi}$ has real Fourier coefficients, then $\pi_{\psi}$ is self-dual. This means we have the isobaric decomposition $\pi_{\psi} \otimes \pi_{\psi} = 1 \boxplus \mathrm{ad}(\pi_{\psi})$. Consequently, $$ \pi_{\psi} \otimes \pi_{\psi} \otimes \pi_{\phi} = (1 \boxplus \mathrm{ad}(\pi_{\psi})) \otimes \pi_{\phi} = \pi_{\phi} \boxplus (\mathrm{ad}(\pi_{\psi}) \otimes \pi_{\phi}),$$ so that we only need to find the Langlands parameters of $\mathrm{ad}(\pi_{\psi}) \otimes \pi_{\phi}$. Let $W_{\mathbb{R}}$ be the Weil group of $\mathbb{R}$, which can be explicitly described as $$W_{\mathbb{R}} = \mathbb{C}^{\times} \cup j\mathbb{C}^{\times},$$ where $j^2 = -1$ and $j z j^{-1} = \overline{z}$ for all $z\in \mathbb{C}^{\times}$. We adopt the notation in \cite[\S 3]{KnappLLCArchimedean} for representations of the Weil group $W_{\mathbb{R}}$. By \cite[Prop.~5.21]{GelbartAutFormsAdeleGps}, the local representations at $\infty$ of $\pi_{\psi}$ and $\pi_{\phi}$ correspond under the local Langlands correspondence to the representations $(0,2\ell_{\psi}+1)_{\mathbb{R}}^{2}$ and $(0,2\ell_{\phi}+1)_{\mathbb{R}}^{2}$ of $W_{\mathbb{R}}$, respectively. It follows by the properties of tensor products of representations of $W_{\mathbb{R}}$ described in \cite[\S 3.1]{WatsonPhD} that the local representation at $\infty$ of $\mathrm{ad}(\pi_{\psi})$ corresponds under the local Langlands correspondence to $$ (0,2(2\ell_{\psi}+2)-2)_{\mathbb{R}}^{2} \oplus (0,1)_{\mathbb{R}}^{1} = (0,4\ell_{\psi}+2)_{\mathbb{R}}^{2} \oplus (0,1)_{\mathbb{R}}^{1}.$$ Consequently, the local representation at $\infty$ of $\mathrm{ad}(\pi_{\psi}) \otimes \pi_{\phi}$ corresponds under the local Langlands correspondence to
\begin{align*}
    &\left[ (0,4\ell_{\psi}+2)_{\mathbb{R}}^{2} \oplus (0,1)_{\mathbb{R}}^{1} \right] \otimes (0,2\ell_{\phi}+1)_{\mathbb{R}}^{2} \\
    &= \left[ (0,4\ell_{\psi}+2)_{\mathbb{R}}^{2} \otimes (0,2\ell_{\phi}+1)_{\mathbb{R}}^{2} \right] \oplus \left[ (0,1)_{\mathbb{R}}^{1} \otimes (0,2\ell_{\phi}+1)_{\mathbb{R}}^{2} \right] \\
    &= (0,4\ell_{\psi}+2\ell_{\phi}+3)_{\mathbb{R}}^{2} \oplus (0,4\ell_{\psi} - 2\ell_{\phi}+1)_{\mathbb{R}}^{2} \oplus (0,2\ell_{\phi}+1)_{\mathbb{R}}^{2}.
\end{align*}
The result follows from the definition of the Langlands parameters of representations of the Weil group $W_{\mathbb{R}}$ as described in \cite[\S 3.1]{WatsonPhD}.
\end{proof}

Additionally, we will need to understand the behavior of $L(s,\mathrm{ad}(\pi_{\psi}))$ at $s=1$. For this purpose, we have the following.

\begin{theorem}[{\cite[Thm.~0.2]{AdjointBound}}]\label{AdjointBound}
Let $f$ be a holomorphic newform of weight $k$ and level at most $2$. Then $$ L(1,\mathrm{ad}(f)) \gg_{\varepsilon} k^{-\varepsilon}$$ for any $\varepsilon>0$.
\end{theorem}

\subsection{Selberg--Harish-Chandra transform}

During the course of the proof of the main results, we will need the theory of the Selberg--Harish-Chandra transform of a point-pair invariant defined on the symmetric space $S^2 \cong \mathrm{SO}(3)/\mathrm{SO}(2)$. We give here the important properties we need based on \cite[\S 2.2]{HumphriesRadziwill}; for more on the theory revolving this transform, see \cite[\S 2]{LubPhiSarI} and \cite{Selberg1956}.

\begin{definition}
Let $k : S^2 \times S^2 \rightarrow \mathbb{C}$ be a point-pair invariant. For $\ell\geq 0$, let $f_{\ell}(x,y,z) := P_{\ell}(z)$ be the zonal spherical harmonic about the north pole $p:=(0,0,1)$ of $S^2$, where $P_{\ell}$ is the $\ell$-th Legendre polynomial. We define the \textit{Selberg--Harish-Chandra transform} (or \textit{spherical transform}) of $k$, denoted by $\widehat{k}:\mathbb{Z}_{\geq 0} \rightarrow \mathbb{C}$, by $$ \widehat{k}(\ell) := \int_{S^2} k(p,x) f_{\ell}(x) \: \mathrm{d}\sigma(x).$$
\end{definition}

The essential property of the spherical transform that we will need is the following.

\begin{lemma}[{\cite[Lemma 2.2]{HumphriesRadziwill}}]\label{LemmaIntPointPairFunction}
Let $\phi$ be a spherical harmonic of degree $\ell_{\phi} \geq 0$ and let $k$ be a point-pair invariant of $S^2$. Then $$ \int_{S^2} k(x,y) \phi(y) \: \mathrm{d}\sigma(y) = \widehat{k}(\ell_{\phi}) \phi(x).$$
\end{lemma}

We will need this result for a particular class of point-pair invariants.

\begin{definition}
For $0 < R \leq \pi$, define the point-pair invariant $$ k_{R}(z,w) := \frac{1}{\mathrm{vol}(B_{R}(w))} 1_{B_{R}(w)}(z) = \begin{dcases*}
    \frac{1}{4\pi \sin^2(R/2)} & if $d(z,w)\leq R$, \\
    0 & otherwise.
\end{dcases*} $$ Let $h_R := \widehat{k_R}$ denote its Selberg--Harish-Chandra transform. We define $$ K_{R}(z,w) := \frac{1}{2}\sum_{\gamma \in \mathcal{O}^{\times}} k_{R}(\gamma z, w).$$
\end{definition}

\begin{lemma}\label{IntBallEqualInnerProd}
If $\phi : \mathcal{O}^{\times} \backslash S^2 \rightarrow \mathbb{C}$ is a spherical harmonic of Laplacian eigenvalue $\lambda_{\psi}^2 = \ell_{\psi}(\ell_{\psi}+1)$, then $$ \frac{1}{\mathrm{vol}(B_R)} \int_{B_R(w)} \phi(z) \: \mathrm{d}\sigma(z) = \langle \phi, K_R(\cdot, w) \rangle = h_R(\ell_{\phi}) \phi(w).$$
\end{lemma}

\noindent \textit{Proof.} On one hand, since $\phi$ is invariant under $\mathcal{O}^{\times}$ and $k_R$ takes only real values, by Lemma \ref{LemmaIntPointPairFunction} we have 
\begin{align*}
    \langle \phi, K_R(\cdot, w)\rangle &= \int_{\mathcal{O}^{\times} \backslash S^2} \phi(z) \overline{K_R(z,w)} \: \mathrm{d}\sigma(z) = \frac{1}{2} \sum_{\gamma\in \mathcal{O}^{\times}} \int_{\mathcal{O}^{\times} \backslash S^2} \phi(z) k_R(\gamma z,w) \: \mathrm{d}\sigma(z) \\
    &= \frac{1}{2} \sum_{\gamma\in \mathcal{O}^{\times}} \int_{\mathcal{O}^{\times} \backslash S^2} \phi(\gamma z) k_R(\gamma z,w) \: \mathrm{d}\sigma(z) \\
    &= \int_{S^2} \phi(z) k_R(z,w) \: \mathrm{d}\sigma(z) = h_R(\ell_{\phi}) \phi(w).
\end{align*}
On the other hand, by definition of $k_R$,
\[
\pushQED{\qed} 
\langle \phi, K_R(\cdot, w)\rangle = \int_{S^2} \phi(z) k_R(z,w) \: \mathrm{d}\sigma(z) = \frac{1}{\mathrm{vol}(B_R)} \int_{B_R(w)} \phi(z) \: \mathrm{d}\sigma(z). \qedhere
\popQED 
\]

\begin{lemma}\label{ConvPointPairInvariants}
For $0\leq \rho<R\leq \pi$, the convolution $k_{R}*k_{\rho}$ is nonnegative, bounded by $1/\mathrm{vol}(B_{R})$, and satisfies $$ k_{R} * k_{\rho}(z,w) = \begin{dcases*}
    \frac{1}{\mathrm{vol}(B_{R})} & if $d(z,w) \leq R-\rho$, \\
    0 & if $d(z,w) > R+\rho$.
\end{dcases*} $$ Consequently, 
\begin{equation}
    \frac{\mathrm{vol}(B_{R-\rho})}{\mathrm{vol}(B_R)} k_{R-\rho} * k_{\rho} (z,w) \leq k_{R}(z,w) \leq \frac{\mathrm{vol}(B_{R+\rho})}{\mathrm{vol}(B_{R})} k_{R+\rho} * k_{\rho}(z,w) \label{ConvPointPairInvariants-eq1}
\end{equation}
and 
\begin{equation}
    \frac{\mathrm{vol}(B_{R-\rho})}{12\mathrm{vol}(B_R)} K_{R-\rho} * K_{\rho} (z,w) \leq K_{R}(z,w) \leq \frac{\mathrm{vol}(B_{R+\rho})}{12\mathrm{vol}(B_{R})} K_{R+\rho} * K_{\rho}(z,w). \label{ConvPointPairInvariants-eq2}
\end{equation}
\end{lemma}

\begin{proof}
The convolution $k_{R}*k_{\rho}$ is nonnegative since both $k_{R}$ and $k_{\rho}$ are. Since $k_{\rho}(t,w)\neq 0$ only when $t \in B_{\rho}(w)$, we have $$ k_{R}*k_{\rho} (z,w) = \int_{S^2} k_{R} (z,t) k_{\rho}(t,w) \: \mathrm{d}\sigma(t) \leq \int_{B_{\rho}(w)} \frac{1}{\mathrm{vol}(B_{R})} \cdot \frac{1}{\mathrm{vol}(B_{\rho})} \: \mathrm{d}\sigma(t) = \frac{1}{\mathrm{vol}(B_{R})}. $$ We now prove the first identity in the result. If $d(z,w)> R+\rho$, then the inequalities $d(z,t)\leq R$ and $d(t,w)\leq \rho$ cannot be true simultaneously, and so $k_{R}(z,t) k_{\rho}(t,w)=0$ for all $t\in S^2$. Now, assume that $d(z,w) \leq R-\rho$. If $d(t,w)>\rho$, then $k_{R}(z,t)k_{\rho}(t,w)=0$, so assume otherwise. It follows that $$ d(z,t) \leq d(z,w) + d(w,t) \leq (R-\rho) + \rho = R, $$ so that $k_{R}(z,t)k_{\rho}(t,w) = \frac{1}{\mathrm{vol}(B_{R})} \cdot \frac{1}{\mathrm{vol}(B_{\rho})}$ for all $t\in B_{\rho}(w)$. From here the first identity follows.

The inequalities \eqref{ConvPointPairInvariants-eq1} follow from the previous observations. As for the inequalities \eqref{ConvPointPairInvariants-eq2}, they follow from \eqref{ConvPointPairInvariants-eq1} and from the computation 
\begin{align}
    K_{R}*K_{\rho}(z,w) &= \frac{1}{4} \sum_{\gamma_1 \in \mathcal{O}^{\times}} \sum_{\gamma_2 \in \mathcal{O}^{\times}} \int_{S^2} k_{R}(\gamma_1 z,t) k_{\rho}(\gamma_2 t,w) \: \mathrm{d}\sigma(t) \label{CompKtok} \\
    &= \frac{1}{4} \sum_{\gamma_1 \in \mathcal{O}^{\times}} \sum_{\gamma_2 \in \mathcal{O}^{\times}} k_{R}*k_{\rho}(\gamma_2 \gamma_1 z, w) \nonumber \\
    &= 6 \sum_{\gamma \in \mathcal{O}^{\times}} k_{R}*k_{\rho}(\gamma z,w), \nonumber
\end{align}
where in the last line we used that, for a fixed $\gamma \in \mathcal{O}^{\times}$, the equation $\gamma = \gamma_2 \gamma_1$ has $|\mathcal{O}^{\times}| = 24$ pairs of solutions $(\gamma_1, \gamma_2)$ in $\mathcal{O}^{\times}$.
\end{proof}

Additionally, we have the following spectral expansion of the convolutions $K_{R}*K_{\rho}$.

\begin{lemma}\label{SpecExpPointPairInvariants}
We have the spectral expansion $$ K_{R}*K_{\rho}(z,w) = \frac{12}{\mathrm{vol}(\mathcal{O}^{\times} \backslash S^2)} + 12\sum_{\phi \in \mathcal{B}} h_{R}(\ell_{\phi}) h_{\rho}(\ell_{\phi}) \phi(z) \overline{\phi(w)}, $$ which converges absolutely and uniformly.
\end{lemma}

\begin{proof}
Since $\mathcal{B}$ is an orthonormal basis of $L^2(\mathcal{O}^{\times} \backslash S^2)$ and $K_{R}*K_{\rho}$ is defined on $\mathcal{O}^{\times} \backslash S^2$, we just need to compute the coefficient corresponding to $f_0 = 1/\sqrt{\mathrm{vol}(\mathcal{O}^{\times} \backslash S^2)}$ and to $\phi \in \mathcal{B}$. Since $\phi \in \mathcal{B}$ is $\mathcal{O}^{\times}$-invariant, by \eqref{CompKtok} and Lemma \ref{LemmaIntPointPairFunction} we have that 
\begin{align*}
    \langle \phi, K_{R}*K_{\rho}(\cdot, w)\rangle &= 6 \int_{\mathcal{O}^{\times} \backslash S^2} \sum_{\gamma \in \mathcal{O}^{\times}} k_{R}*k_{\rho}(\gamma z,w) \phi(\gamma z) \: \mathrm{d}\sigma(z) = 12 \int_{S^2} k_{R}*k_{\rho}(z,w) \phi(z) \: \mathrm{d}\sigma(z) \\
    &= 12 \int_{S^2} \int_{S^2} k_{R}(z,t) k_{\rho}(t,w) \phi(z) \: \mathrm{d}\sigma(t) \: \mathrm{d}\sigma(z) \\
    &= 12 h_{R}(\ell_{\phi}) \int_{S^2} k_{\rho}(t,w) \phi(t) \: \mathrm{d}\sigma(t) = 12 h_{R}(\ell_{\phi}) h_{\rho}(\ell_{\phi}) \phi(w).
\end{align*}
Similarly, one can show that $\langle f_0, K_{R}*K_{\rho}(\cdot, w)\rangle = 12$, and so we obtain the desired spectral expansion. The spectral expansion converges absolutely and uniformly since the convolution $K_{R}*K_{\rho}$ is continuous in both variables $z$ and $w$.
\end{proof}

We now give asymptotics for $h_{R}$, which we will need later.

\begin{lemma}[{\cite[Lemmas 2.12, 2.13]{HumphriesRadziwill}}]\label{boundsSHCTransform}
Suppose that $R\leq \pi-\varepsilon$ for some fixed $\varepsilon>0$. For $m\in \mathbb{N}$, we have $$ h_{R}(m) \ll \begin{dcases*}
    1 & for $m\leq \frac{1}{R}$, \\
    \frac{1}{R^{3/2} m^{3/2}} & for $m\geq \frac{1}{R}$.
\end{dcases*} $$
\end{lemma}

\subsection{1-Wasserstein distance and the Berry--Esseen inequality}

To be able to bound the $1$-Wasserstein distance between the relevant measures in our problem, we will need the Berry--Esseen inequality, which gives upper bounds for the $1$-Wasserstein distance in terms of the Fourier transforms of the involved measures.

This inequality was originally proved for the real line by Esseen in \cite[Ch.~II, Thm.~2b]{Esseen-Wasserstein}, but the main ideas to prove it had already appeared in the earlier paper \cite{Berry-Wasserstein} by Berry. The analogous inequality for the $n$-sphere was later shown in \cite[Thm.~1]{GrabnerTichy-Wasserstein} by Grabner and Tichy. In the last few years, such an inequality was proved for the $n$-torus $\mathbb{T}^{n} = (\mathbb{R}/\mathbb{Z})^n$ by Bobkov and Ledoux in \cite[Prop.~2]{BobkovLedoux-Wasserstein} and independently by Borda in \cite[Prop.~3]{Borda-Wasserstein}. Recently, it was shown in a greater generality for compact Lie groups, compact homogeneous spaces, and compact Riemannian manifolds in \cite{BordaCuenin-Wasserstein} by Borda and Cuenin, and then extended to finite-volume (not necessarily compact) hyperbolic surfaces in \cite[Thms.~1.5, 1.10]{Humphries2025} by Humphries. The version we require for our purposes is the following, which is a consequence of \cite[Thm.~7]{BordaCuenin-Wasserstein} and \cite[Thm.~1]{GrabnerTichy-Wasserstein}.

\begin{proposition}\label{BerryEsseenSphere}
Let $\sigma'$ be a Borel probability measure on $\mathcal{O}^{\times} \backslash S^2$. Then, for any integer $T\geq 1$, $$ \mathscr{W}_{1}\left( \sigma', \frac{1}{\sigma(\mathcal{O}^{\times} \backslash S^2)} \sigma \right) \ll \frac{1}{T} + \left( \sum_{\substack{\phi \in \mathcal{B} \\ \ell_{\phi} \leq T }} \frac{1}{\ell_{\phi}(\ell_{\phi}+1)} \left| \int_{\mathcal{O}^{\times} \backslash S^2} \phi(x) \: \mathrm{d}\sigma'(x) \right|^{2} \right)^{\frac{1}{2}}. $$
\end{proposition}

\section{Proofs}\label{Proofs}

In this section we show the main results of this paper, namely Theorems \ref{MainThm-Wass}, \ref{MainThm-AllCenters}, \ref{MainThm-SphCapDisc}, and \ref{MainThm-AlmostAllCenters}. We start with Theorem \ref{MainThm-Wass}.

\begin{proof}[Proof of Theorem \ref{MainThm-Wass}.]
Fix $T\geq 1$. By Corollary \ref{WatsonIchinoClassical}, Proposition \ref{LfuncsData}, Theorem \ref{AdjointBound}, GLH, and Proposition \ref{BerryEsseenSphere}, we have 
\begin{align}
    \mathscr{W}_{1}\left(\sigma_{\psi}, \frac{1}{\sigma(\mathcal{O}^{\times} \backslash S^2)} \sigma \right) &\ll \frac{1}{T} + \left( \sum_{\substack{\phi \in \mathcal{B} \\ \ell_{\phi} \leq T }} \frac{1}{\ell_{\phi}(\ell_{\phi}+1)} \left| \int_{\mathcal{O}^{\times} \backslash S^2} \phi(x) \: \mathrm{d}\sigma_{\psi}(x) \right|^{2} \right)^{\frac{1}{2}} \label{MainThmWass-eq1} \\
    &= \frac{1}{T} + \left( \sum_{\substack{\phi \in \mathcal{B} \\ \ell_{\phi} \leq T }} \frac{1}{\ell_{\phi}(\ell_{\phi}+1)} |\langle \phi, |\psi|^2\rangle |^{2} \right)^{\frac{1}{2}} \nonumber \\
    &\ll_{\varepsilon} \frac{1}{T} + \left( \ell_{\psi}^{-\frac{1}{2} + \varepsilon} \sum_{\substack{\phi \in \mathcal{B} \\ \ell_{\phi} \leq \min\{T,2\ell_{\psi}\}}} \frac{1}{\ell_{\phi}^3 (1+2\ell_{\psi} - \ell_{\phi})^{1/2}} \right)^{\frac{1}{2}}. \nonumber
\end{align}
We choose $T = \left\lceil \ell_{\psi}^{\frac{1}{2}} \right\rceil$. We divide the sum on the right-hand side of \eqref{MainThmWass-eq1} into dyadic ranges of the form $a_m \leq \ell_{\phi} \leq 2a_m$, where $a_m := 2^m$ and $0\leq m\leq M := \lfloor \log_2(T) \rfloor$. For each dyadic range, since $\ell_{\phi} \asymp a_m$, then $1+2\ell_{\psi}-a_m \asymp \ell_{\psi}$, so that
\begin{equation}
    \sum_{\substack{\phi \in \mathcal{B} \\ a_m \leq \ell_{\phi} \leq 2a_m}} \frac{1}{\ell_{\phi}^3 (1+2\ell_{\psi} - \ell_{\phi})^{1/2}} \asymp \frac{1}{a_{m}^{3} \ell_{\psi}^{1/2}} \sum_{\substack{\phi \in \mathcal{B} \\ a_m \leq \ell_{\phi} \leq 2a_m}} 1 \ll \frac{1}{a_m \ell_{\psi}^{1/2}}, \label{MainThmWass-eq2}
\end{equation}
where in the last inequality we used the Weyl law. Since the sequence $\{a_m\}$ is increasing and $M \ll_{\varepsilon} \ell_{\psi}^{\varepsilon}$, by \eqref{MainThmWass-eq1} and \eqref{MainThmWass-eq2} it follows that 
\begin{align*}
    \mathscr{W}_{1}\left(\sigma_{\psi}, \frac{1}{\sigma(\mathcal{O}^{\times} \backslash S^2)} \sigma \right) &\ll_{\varepsilon} \ell_{\psi}^{-\frac{1}{2}} + \left( \ell_{\psi}^{-\frac{1}{2}+\varepsilon} \sum_{m=0}^{M} \frac{1}{a_{m} \ell_{\psi}^{1/2}} \right)^{\frac{1}{2}} \\
    &\ll \ell_{\psi}^{-\frac{1}{2}} + \left( \ell_{\psi}^{-\frac{1}{2}+\varepsilon} \frac{M}{a_{0}\ell_{\psi}^{1/2}} \right)^{\frac{1}{2}} \ll_{\varepsilon} \ell_{\psi}^{-\frac{1}{2}+\varepsilon}. \qedhere
\end{align*}
\end{proof}

We continue with Theorem \ref{MainThm-AllCenters}; for the idea of the proof, we follow \cite[Thm.~1.5]{HumphriesRadziwill}.

\vspace{0.3cm}

\noindent \textit{Proof of Theorem \ref{MainThm-AllCenters}.} We will show that there exist functions $g_{R,w}(\ell_{\psi})$ and $h_{R,w}(\ell_{\psi})$ such that $$ g_{R,w}(\ell_{\psi}) \leq \frac{1}{\mathrm{vol}(B_R)} \int_{B_R (w)} |\psi(z)|^2 \: \mathrm{d}\sigma(z) - \frac{1}{\mathrm{vol}(\mathcal{O}^{\times} \backslash S^2)} \leq h_{R,w}(\ell_{\psi})$$ and that $|g_{R,w}(\ell_{\psi})|, |h_{R,w}(\ell_{\psi})| = o_{w,\delta}(1)$. This claim implies the result. We will only focus on the lower bound; the upper bound works in a similar way. Let $\varepsilon>0$ be sufficiently small and fix $0<\delta < \frac{1}{3}$. From Lemmas \ref{SpectralRes}, \ref{IntBallEqualInnerProd}, \ref{ConvPointPairInvariants}, and \ref{SpecExpPointPairInvariants}, for $0<\rho < R$ we have 
\begin{align}
    &\frac{1}{\mathrm{vol}(B_R)} \int_{B_R (w)} |\psi(z)|^2 \: \mathrm{d}\sigma(z) \label{MainThmAllCenters-eq0} \\
    &\geq \frac{\mathrm{vol}(B_{R-\rho})}{12\mathrm{vol}(B_R)} \langle |\psi|^2, K_{R-\rho} * K_{\rho} (\cdot, w) \rangle \nonumber \\
    &= \frac{\mathrm{vol}(B_{R-\rho})}{\mathrm{vol}(B_{R})\mathrm{vol}(\mathcal{O}^{\times} \backslash S^2)} + \frac{\mathrm{vol}(B_{R-\rho})}{\mathrm{vol}(B_R)} \sum_{\phi \in \mathcal{B}} h_{R-\rho}(\ell_{\phi}) h_{\rho}(\ell_{\phi}) \langle |\psi|^2, \phi\rangle \overline{\phi(w)}. \nonumber
\end{align}
We impose the condition $\ell_{\psi}^{-1} \leq \rho \leq \frac{R}{2}$. We will show that 
\begin{equation}
    \left|\frac{\mathrm{vol}(B_{R-\rho})}{\mathrm{vol}(B_R)} \sum_{\phi \in \mathcal{B}} h_{R-\rho}(\ell_{\phi}) h_{\rho}(\ell_{\phi}) \langle |\psi|^2, \phi\rangle \overline{\phi(w)} \right| \ll_{\varepsilon, \delta} \ell_{\psi}^{-\frac{1}{2} + \frac{3}{2}\delta + \varepsilon} + R^{-\frac{3}{2}} \ell_{\psi}^{- \frac{1}{2} + \varepsilon} + R^{-\frac{3}{2}} \rho^{-\frac{3}{2}} \ell_{\psi}^{-2+\varepsilon}. \label{MainThmAllCenters-eq1}
\end{equation}
Denote by $I$ the left-hand side of \eqref{MainThmAllCenters-eq1}. By Corollary \ref{WatsonIchinoClassical}, we know that the terms in the sum defining $I$ are nonzero only when $\ell_{\phi} \leq 2\ell_{\psi}$ (and additionally $\ell_{\phi}$ must be even, but we will not use this in the proof). We divide this sum into four ranges: $\ell_{\phi} \leq R^{-1}$, $R^{-1} \leq \ell_{\phi} \leq \rho^{-1}$, $\rho^{-1} \leq \ell_{\phi} \leq \ell_{\psi}$, and $\ell_{\psi} \leq \ell_{\phi} \leq 2\ell_{\phi}$, so that by using the triangle inequality, the facts that $\mathrm{vol}(B_{r}) \asymp r^2$ for $r$ sufficiently small and $R-\rho \asymp R$, and Lemma \ref{boundsSHCTransform}, we have
\begin{align}
    |I| &\ll \sum_{\substack{\phi \in \mathcal{B} \\ \ell_{\phi} \leq R^{-1}}} |\langle |\psi|^2, \phi\rangle | \: |\phi(w)| + \frac{1}{R^{\frac{3}{2}}} \sum_{\substack{\phi \in \mathcal{B} \\ R^{-1} \leq \ell_{\phi} \leq \rho^{-1}}} \frac{1}{\ell_{\phi}^{3/2}} |\langle |\psi|^2, \phi\rangle | \: |\phi(w)| \label{MainThmAllCenters-eq2} \\
    &\quad + \frac{1}{R^{\frac{3}{2}} \rho^{\frac{3}{2}}} \sum_{\substack{\phi \in \mathcal{B} \\ R^{-1} \leq \ell_{\phi} \leq \ell_{\psi}}} \frac{1}{\ell_{\phi}^3} |\langle |\psi|^2, \phi\rangle | \: |\phi(w)| + \frac{1}{R^{\frac{3}{2}} \rho^{\frac{3}{2}}} \sum_{\substack{\phi \in \mathcal{B} \\ \ell_{\psi} \leq \ell_{\phi} \leq 2\ell_{\psi}}} \frac{1}{\ell_{\phi}^3} |\langle |\psi|^2, \phi\rangle | \: |\phi(w)| \nonumber
\end{align}

We first deal with the range $\ell_{\phi} \leq R^{-1}$. We subdivide this range into dyadic ranges of the form $a_{m} \leq \ell_{\phi} \leq 2a_{m}$, where $a_{m} := 2^{m}$ and $0\leq m \leq M_1$, with $M_1 := \lfloor \log_{2}(R^{-1}) \rfloor$. For each dyadic range, since $\ell_{\phi} \asymp a_{m}$ and $1+2\ell_{\psi} - a_m \asymp \ell_{\psi}$, by using the Cauchy--Schwarz inequality, the local Weyl law, the Weyl law, Corollary \ref{WatsonIchinoClassical}, Lemma \ref{LfuncsData}, Theorem \ref{AdjointBound}, and GLH, we have 
\begin{align}
    \sum_{\substack{\phi \in \mathcal{B} \\ a_{m} \leq \ell_{\phi} \leq 2a_{m}}} |\langle |\psi|^2, \phi\rangle | \: |\phi(w)| &\ll \left( \sum_{\substack{\phi \in \mathcal{B} \\ a_{m} \leq \ell_{\phi} \leq 2a_{m}}} |\langle |\psi|^2, \phi \rangle|^2 \right)^{\frac{1}{2}} \left( \sum_{\substack{\phi \in \mathcal{B} \\ a_{m} \leq \ell_{\phi} \leq 2a_{m}}} |\phi(w)|^2 \right)^{\frac{1}{2}} \label{MainThmAllCenters-eq3} \\
    &\ll_{\varepsilon} a_{m} \left( \ell_{\psi}^{-\frac{1}{2} + 2\varepsilon} \sum_{\substack{\phi \in \mathcal{B} \\ a_{m} \leq \ell_{\phi} \leq 2a_{m}}} \frac{1}{\ell_{\phi} (1+2\ell_{\psi} - \ell_{\phi})^{\frac{1}{2}}} \right)^{\frac{1}{2}} \nonumber \\
    &\asymp_{\varepsilon} \frac{a_{m}^{\frac{1}{2}}}{\ell_{\psi}^{\frac{1}{4}-\varepsilon}} \left( \sum_{\substack{\phi \in \mathcal{B} \\ a_{m} \leq \ell_{\phi} \leq 2a_{m} }} 1 \right)^{\frac{1}{2}} \ll_{\varepsilon} a_{m}^{\frac{3}{2}} \ell_{\psi}^{-\frac{1}{2} + \varepsilon}. \nonumber
\end{align}
Since $M_1 \ll_{\varepsilon} \ell_{\psi}^{\varepsilon}$, it follows that 
\begin{equation}
    \sum_{\substack{\phi \in \mathcal{B} \\ \ell_{\phi} \leq R^{-1}}} |\langle |\psi|^2, \phi\rangle | \: |\phi(w)| \leq \sum_{m=0}^{M_1} \sum_{\substack{\phi \in \mathcal{B} \\ a_{m} \leq \ell_{\phi} \leq 2a_{m}}} |\langle |\psi|^2, \phi\rangle | \: |\phi(w)| \ll_{\varepsilon} \ell_{\psi}^{\varepsilon} a_{M_1}^{\frac{3}{2}} \ell_{\psi}^{-\frac{1}{2} + \varepsilon} \ll_{\varepsilon} \ell_{\psi}^{-\frac{1}{2} + \frac{3}{2}\delta + 2\varepsilon}. \label{MainThmAllCenters-eq4}
\end{equation}

We deal with the range $R^{-1} \leq \ell_{\phi} \leq \rho^{-1}$ in a similar way. We subdivide this range into dyadic ranges of the form $b_{m} \leq \ell_{\phi} \leq 2b_{m}$, where $b_{m} := 2^{m}R^{-1}$ and $0\leq m \leq M_2 := \lfloor \log_{2}(\rho^{-1} R) \rfloor$. Doing the same analysis as in \eqref{MainThmAllCenters-eq3} and \eqref{MainThmAllCenters-eq4}, we obtain 
\begin{equation}
    \sum_{\substack{\phi \in \mathcal{B} \\ R^{-1} \leq \ell_{\phi} \leq \rho^{-1}}} \frac{1}{\ell_{\phi}^{3/2}} |\langle |\psi|^2, \phi\rangle | \: |\phi(w)| \ll_{\varepsilon} \ell_{\psi}^{- \frac{1}{2} + 2\varepsilon}. \label{MainThmAllCenters-eq5}
\end{equation}

Similarly with the range $\rho^{-1} \leq \ell_{\phi} \leq \ell_{\psi}$, we subdivide the range into dyadic ranges of the form $c_m \leq \ell_{\phi} \leq 2c_m$, where $c_m := 2^{m} \rho^{-1}$ and $0\leq m\leq M_3 := \lfloor \log_{2}(\ell_{\psi} \rho ) \rfloor$. The same procedure yields 
\begin{equation}
    \sum_{\substack{\phi \in \mathcal{B} \\ \rho^{-1} \leq \ell_{\phi} \leq \ell_{\psi}}} \frac{1}{\ell_{\phi}^{3}} |\langle |\psi|^2, \phi\rangle | \: |\phi(w)| \ll_{\varepsilon, \delta} \ell_{\psi}^{-\frac{1}{2} + 2\varepsilon} \rho^{\frac{3}{2}}. \label{MainThmAllCenters-eq6}
\end{equation}

Finally, for the range $\ell_{\psi} \leq \ell_{\phi} \leq 2\ell_{\psi}$, we subdivide this range into dyadic ranges of the form $2\ell_{\psi} - 2d_{m} \leq \ell_{\phi} \leq 2\ell_{\psi} - d_{m}$, where $d_{m} := 2^{m}$ and $0\leq m \leq M_4 := \lfloor \log_{2}(\ell_{\psi}) \rfloor$. From following a similar analysis as before we obtain
\begin{equation}
    \sum_{\substack{\phi \in \mathcal{B} \\ \ell_{\psi} \leq \ell_{\phi} \leq 2\ell_{\psi}}} \frac{1}{\ell_{\phi}^3} |\langle |\psi|^2, \phi\rangle | \: |\phi(w)| \ll_{\varepsilon} \ell_{\psi}^{-2+2\varepsilon}. \label{MainThmAllCenters-eq7}
\end{equation}
Combining \eqref{MainThmAllCenters-eq2}, \eqref{MainThmAllCenters-eq4}, \eqref{MainThmAllCenters-eq5}, \eqref{MainThmAllCenters-eq6}, and \eqref{MainThmAllCenters-eq7} we obtain \eqref{MainThmAllCenters-eq1}. 

Now, using that $\mathrm{vol}(B_{R-\rho}) = R^2 + \mathcal{O}(R\rho)$, from \eqref{MainThmAllCenters-eq0} and \eqref{MainThmAllCenters-eq1} we have 
\begin{align*}
    &\frac{1}{\mathrm{vol}(B_R)} \int_{B_{R}(w)} |\psi(z)|^2 \: \mathrm{d}\sigma(z) \\
    &\geq \frac{1}{\mathrm{vol}(\mathcal{O}^{\times} \backslash S^2)} + \mathcal{O}_{\varepsilon, \delta}\left( \frac{\rho}{R} + \ell_{\psi}^{-\frac{1}{2} + \frac{3}{2}\delta + \varepsilon} + R^{-\frac{3}{2}} \ell_{\psi}^{- \frac{1}{2} + \varepsilon} + R^{-\frac{3}{2}} \rho^{-\frac{3}{2}} \ell_{\psi}^{-2+\varepsilon} \right).
\end{align*}
Taking $\rho = R^{-\frac{1}{5}} \ell_{\psi}^{-\frac{4}{5}}$ (which satisfies that $\ell_{\psi}^{-1} \leq \rho \leq R/2$) and using that $R \gg \ell_{\psi}^{-\delta}$, we have that $$ \frac{\rho}{R} \ll \ell_{\psi}^{-\frac{2}{5}(2-3\delta)} \: \text{ and } \: R^{-\frac{3}{2}} \rho^{-\frac{3}{2}} \ell_{\psi}^{-2} \ll \ell_{\psi}^{-\frac{2}{5}(2-3\delta)},$$ from where it follows that 
\[
\pushQED{\qed} 
\frac{1}{\mathrm{vol}(B_R)} \int_{B_{R}(w)} |\psi(z)|^2 \: \mathrm{d}\sigma(z) \geq \frac{1}{\mathrm{vol}(\mathcal{O}^{\times} \backslash S^2)} + \mathcal{O}_{\varepsilon, \delta}\left( \ell_{\psi}^{-\frac{1}{2}(1-3\delta) + \varepsilon} \right). \qedhere
\popQED 
\]

The proof of Theorem \ref{MainThm-SphCapDisc} follows a similar idea as the proof of Theorem \ref{MainThm-AllCenters}.

\vspace{0.3cm}

\noindent \textit{Proof of Theorem \ref{MainThm-SphCapDisc}.} As in the proof of Theorem \ref{MainThm-AllCenters}, we will give a lower bound for $\sigma_{\psi}(B_{R}(w)) - \sigma(B_{R}(w))/\sigma(\mathcal{O}^{\times} \backslash S^2)$ that will have the order $\ell_{\psi}^{-\frac{1}{2}+\varepsilon}$. The corresponding upper bound (with the same order of growth) can be shown similarly.

Fix $0<R\leq \pi$ and $w\in \mathcal{O}^{\times} \backslash S^2$. We know that $\sigma(B_r(w)) \asymp r^2$ and that $\sigma(B_{R-\rho}(w)) = \sigma(B_{R}(w)) + \mathcal{O}(R\rho)$. From these observations, the triangle inequality, and \eqref{MainThmAllCenters-eq0}, for $0<\rho<R/2$ we have that
\begin{equation}
    \sigma_{\psi}(B_{R}(w)) \geq \frac{\sigma(B_{R}(w))}{\sigma(\mathcal{O}^{\times} \backslash S^2)} + \mathcal{O}\left( R\rho + R^2 \sum_{\phi \in \mathcal{B}} |h_{R-\rho}(\ell_{\phi}) h_{\rho}(\ell_{\phi})| \: |\langle |\psi|^2, \phi \rangle| \: |\phi(w)| \right). \label{ProofMainThmSphCapDisc-eq1}
\end{equation}

We only consider the case $R\geq 2\ell_{\psi}^{-1}$; the other case works similarly. Assuming the condition $\ell_{\psi}^{-1} \leq \rho \leq R/2$, from the proof of \eqref{MainThmAllCenters-eq1} we have that
\begin{equation}
    R^2 \sum_{\phi \in \mathcal{B}} |h_{R-\rho}(\ell_{\phi}) h_{\rho}(\ell_{\phi})| \: |\langle |\psi|^2, \phi \rangle| \: |\phi(w)| \ll_{\varepsilon} R^{\frac{1}{2}} \ell_{\psi}^{-\frac{1}{2}+\varepsilon} + R^{\frac{1}{2}} \rho^{-\frac{3}{2}} \ell_{\psi}^{-2+\varepsilon}. \label{ProofMainThmSphCapDisc-eq2}
\end{equation}
for all $\varepsilon>0$. Again, this can be shown by dividing the left-hand side sum into four ranges ($\ell_{\phi} \leq R^{-1}$, $R^{-1} \leq \ell_{\phi} \leq \rho^{-1}$, $\rho^{-1} \leq \ell_{\phi} \leq \ell_{\psi}$, and $\ell_{\psi} \leq \ell_{\phi} \leq 2\ell_{\psi}$) and analyzing each range as in the proof of Theorem \ref{MainThm-AllCenters}. 

We choose $$\rho = \begin{dcases*}
    \ell_{\psi}^{-\frac{1}{2}} & if $R \geq 2\ell_{\psi}^{-\frac{1}{2}}$, \\
    R/2 & otherwise.
\end{dcases*} $$ If $R\geq 2\ell_{\psi}^{-\frac{1}{2}}$, inserting this in \eqref{ProofMainThmSphCapDisc-eq1} and \eqref{ProofMainThmSphCapDisc-eq2}, and using that $R\ll 1$, it follows that $$\sigma_{\psi}(B_{R}(w)) \geq \frac{\sigma(B_{R}(w))}{\sigma(\mathcal{O}^{\times} \backslash S^2)} + \mathcal{O}_{\varepsilon} \left( \ell_{\psi}^{-\frac{1}{2}} + \ell_{\psi}^{-\frac{1}{2}+\varepsilon} + \left( \ell_{\psi}^{-\frac{1}{2}} \right)^{-\frac{3}{2}} \ell_{\psi}^{-2+\varepsilon} \right) = \frac{\sigma(B_{R}(w))}{\sigma(\mathcal{O}^{\times} \backslash S^2)} + \mathcal{O}_{\varepsilon} \left( \ell_{\psi}^{-\frac{1}{2} + \varepsilon} \right). $$ Otherwise, using that $2\ell_{\psi}^{-1} \leq R \leq 2\ell_{\psi}^{-\frac{1}{2}}$, we obtain
\[
\pushQED{\qed} 
\sigma_{\psi}(B_{R}(w)) \geq \frac{\sigma(B_{R}(w))}{\sigma(\mathcal{O}^{\times} \backslash S^2)} + \mathcal{O}_{\varepsilon} \left( R^2 + R^{\frac{1}{2}} \ell_{\psi}^{-\frac{1}{2} + \varepsilon} + R^{-1} \ell_{\psi}^{-2+\varepsilon} \right) = \frac{\sigma(B_{R}(w))}{\sigma(\mathcal{O}^{\times} \backslash S^2)} + \mathcal{O}_{\varepsilon} \left( \ell_{\psi}^{-\frac{3}{4} + \varepsilon} \right). \qedhere
\popQED 
\]

Finally, we proceed to prove Theorem \ref{MainThm-AlmostAllCenters}; we follow \cite[\S 5.1]{HumphriesPhD}. The result will follow from the next proposition.

\begin{proposition}\label{BoundsVariance}
Let $\psi$ be a spherical harmonic of Laplacian eigenvalue $\lambda_{\psi}^2 = \ell_{\psi}(\ell_{\psi}+1)$ for some integer $\ell_{\psi}\geq 0$, eigenfunction of all the Hecke operators, invariant under $\mathcal{O}^{\times}$, and normalized such that $\langle \psi, \psi\rangle = 1$. For $R>0$, let $$\mathrm{Var}(\psi; R) := \int_{\mathcal{O}^{\times} \backslash S^2} \left( \frac{1}{\mathrm{vol}(B_R)} \int_{B_R(w)} |\psi(z)|^2 \: \mathrm{d}\sigma(z) - \frac{1}{\mathrm{vol}(\mathcal{O}^{\times} \backslash S^2)} \right)^2 \mathrm{d}\sigma(w).$$ Assume the generalized Lindelöf hypothesis, and suppose that $R \gg \ell_{\psi}^{-\delta}$ for some $\delta >0$. Then, for $0<\delta < 1$, $$ \mathrm{Var}(\psi; R) \ll_{\varepsilon, \delta} \ell_{\psi}^{-(1-\delta)+\varepsilon}$$ as $\ell_{\psi}\rightarrow \infty$.
\end{proposition}

\noindent \textit{Proof of Theorem \ref{MainThm-AlmostAllCenters}.} By Chebyshev's inequality and Proposition \ref{BoundsVariance}, we have 
\begin{align*}
    &\mathrm{vol}\left( \left\{ w\in \mathcal{O}^{\times} \backslash S^2 : \left| \frac{1}{\mathrm{vol}(B_R)} \int_{B_R(w)} |\psi(z)|^2 \: \mathrm{d}\sigma(z) - \frac{1}{\mathrm{vol}(\mathcal{O}^{\times} \backslash S^2)} \right| > c \right\} \right) \\
    &\leq \frac{1}{c^2} \mathrm{Var}(\psi; R) \ll_{\varepsilon, \delta} \frac{\ell_{\psi}^{-(1-\delta)+\varepsilon}}{\left(\ell_{\psi}^{-(1-\delta)/2 + \varepsilon} \right)^2} = \ell_{\psi}^{-\varepsilon},
\end{align*}
which implies the result. \hfill \qed

\vspace{0.1cm}
Then Theorem \ref{MainThm-AlmostAllCenters} follows from Proposition \ref{BoundsVariance}. To prove Proposition \ref{BoundsVariance}, we first give a spectral expansion of $\mathrm{Var}(\psi,R)$.

\begin{proposition}\label{VarianceFormula}
Let $\psi$ be an $L^2$-normalized spherical harmonic of Laplacian eigenvalue $\lambda_{\psi}^2 = \ell_{\psi}(\ell_{\psi}+1)$ for some integer $\ell_{\psi}>0$. Then, for $R>0$, $$\mathrm{Var}(\psi;R) = \sum_{\phi \in \mathcal{B}} |h_R(\ell_{\phi})|^2 \left| \langle |\psi|^2, \phi \rangle \right|^2.$$
\end{proposition}

\begin{proof}
By Lemmas \ref{SpectralRes} and \ref{IntBallEqualInnerProd}, we have 
\begin{align*}
    \langle |\psi|^2, K_R(\cdot, w) \rangle &= \langle |\psi|^2, f_0\rangle \langle f_0, K_R(\cdot, w)\rangle + \sum_{\phi \in \mathcal{B}} \langle |\psi|^2, \phi\rangle \langle \phi, K_R(\cdot, w)\rangle \\
    &= \frac{1}{\mathrm{vol}(\mathcal{O}^{\times} \backslash S^2 )} + \sum_{\phi \in \mathcal{B}} h_R(\ell_{\phi}) \phi(w) \langle |\psi|^2 , \phi \rangle.
\end{align*}
This equation implies, again by Lemma \ref{IntBallEqualInnerProd}, that $$ \frac{1}{\mathrm{vol}(B_R)} \int_{B_R(w)} |\psi(z)|^2 \: \mathrm{d}\sigma(z) - \frac{1}{\mathrm{vol}(\mathcal{O}^{\times} \backslash S^2)} = \sum_{\phi \in \mathcal{B}} h_R(\ell_{\phi}) \phi(w) \langle |\psi|^2 , \phi \rangle.$$ Since $\mathcal{B}$ is an orthonormal basis of $L^2(\mathcal{O}^{\times} \backslash S^2)$, squaring and integrating over $w$ we have
\begin{align*}
    \mathrm{Var}(\psi;R) &= \int_{\mathcal{O}^{\times} \backslash S^2} \left| \sum_{\phi \in \mathcal{B}} h_R(\ell_{\phi}) \phi(w) \langle |\psi|^2 , \phi \rangle \right|^2 \mathrm{d}\sigma(w) \\
    &= \int_{\mathcal{O}^{\times} \backslash S^2} \sum_{\phi_1, \phi_2 \in \mathcal{B}} h_R(\ell_{\phi_1}) \phi_1(w) \langle |\psi|^2 , \phi_1 \rangle \overline{h_R(\ell_{\phi_2}) \phi_2(w) \langle |\psi|^2 , \phi_2 \rangle} \: \mathrm{d}\sigma(w) \\
    &= \sum_{\phi_1, \phi_2 \in \mathcal{B}} h_R(\ell_{\phi_1}) \overline{h_R(\ell_{\phi_2})}  \langle |\psi|^2 , \phi_1 \rangle \overline{\langle |\psi|^2 , \phi_2 \rangle} \int_{\mathcal{O}^{\times} \backslash S^2} \phi_1(w) \overline{\phi_2(w)} \: \mathrm{d}\sigma(w) \\
    &= \sum_{\phi \in \mathcal{B}} |h_R(\ell_{\phi})|^2 \left| \langle |\psi|^2, \phi \rangle \right|^2. \qedhere
\end{align*}
\end{proof}

\begin{proof}[Proof of Proposition \ref{BoundsVariance}.]
By Corollary \ref{WatsonIchinoClassical}, Lemma \ref{LfuncsData}, Theorem \ref{AdjointBound}, GLH, and Proposition \ref{VarianceFormula}, we have that $$\mathrm{Var}(\psi;R) \ll_{\varepsilon} \ell_{\psi}^{-\frac{1}{2}+\varepsilon} \sum_{\substack{\phi \in \mathcal{B} \\ \ell_{\phi} \leq 2\ell_{\psi}}} \frac{|h_R(\ell_{\phi})|^2}{\ell_{\phi} (1+2\ell_{\psi} - \ell_{\phi})^{1/2}}.$$ We now separate the double sum into three ranges: $\ell_{\phi} \leq R^{-1}$, $R^{-1} \leq \ell_{\phi} \leq \ell_{\psi}$, and $\ell_{\psi} \leq \ell_{\phi} \leq 2\ell_{\psi}$. By Lemma \ref{boundsSHCTransform}, we obtain
\begin{align}
    \mathrm{Var}(\psi;R) &\ll_{\varepsilon} \ell_{\psi}^{-\frac{1}{2}+3\varepsilon} \sum_{\substack{\phi \in \mathcal{B} \\ \ell_{\phi} \leq R^{-1}}} \frac{1}{\ell_{\phi} (1+2\ell_{\psi} - \ell_{\phi})^{1/2}} + \ell_{\psi}^{3\delta-\frac{1}{2}+3\varepsilon} \sum_{\substack{\phi \in \mathcal{B} \\ R^{-1} \leq \ell_{\phi} \leq \ell_{\psi}}} \frac{1}{\ell_{\phi}^{4} (1+2\ell_{\psi} - \ell_{\phi})^{1/2}} \label{BoundsVariance-small-eq1} \\
    &\quad + \ell_{\psi}^{3\delta-\frac{1}{2}+3\varepsilon} \sum_{\substack{\phi \in \mathcal{B} \\ \ell_{\psi} \leq \ell_{\phi} \leq 2\ell_{\psi}}} \frac{1}{\ell_{\phi}^{4} (1+2\ell_{\psi} - \ell_{\phi})^{1/2}}. \nonumber
\end{align}

For the first sum on the right-hand side of \eqref{BoundsVariance-small-eq1}, we have that $1+2\ell_{\psi} - \ell_{\psi} \asymp \ell_{\psi}$ since $\ell_{\phi} \leq R^{-1} \ll \ell_{\psi}$. Hence, by the Weyl law, we obtain 
\begin{equation}
    \sum_{\substack{\phi \in \mathcal{B} \\ \ell_{\phi} \leq R^{-1}}} \frac{1}{\ell_{\phi} (1+2\ell_{\psi} - \ell_{\phi})^{1/2}} \ll \ell_{\psi}^{-\frac{1}{2}} \sum_{\ell \leq R^{-1}} 1 \ll R^{-1} \ell_{\psi}^{-\frac{1}{2}} \ll \ell_{\psi}^{\delta - \frac{1}{2}}. \label{BoundsVariance-small-eq2}
\end{equation}

For the second and third sums on the right-hand side of \eqref{BoundsVariance-small-eq1}, we follow the same idea as with the last two ranges in the proof of Theorem \ref{MainThm-AllCenters}: subdivide into dyadic ranges and use the Weyl law. Using that $0<\delta<1$, we obtain that both sums are $\ll_{\varepsilon, \delta} \ell_{\psi}^{-2\delta - \frac{1}{2}+\varepsilon}$. The result follows from combining these bounds together with \eqref{BoundsVariance-small-eq1} and \eqref{BoundsVariance-small-eq2}.
\end{proof}

\bibliographystyle{alpha}
\bibliography{references}

@article{ZelditchQE,
    author = {Zelditch, Steven},
    title = {Uniform distribution of eigenfunctions on compact hyperbolic surfaces},
    journal = {Duke Math. J.},
    year = {1987},
    volume = {55},
    number = {4}
}

@phdthesis{WatsonPhD,
    author = {Watson, Thomas C.},
    title = {Rankin {T}riple {P}roducts and {Q}uantum {C}haos},
    school = {Princeton {U}niversity},
    year = {2002}
}

@article{LindenstraussQUE,
    author = {Lindenstrauss, Elon},
    title = {Invariant {M}easures and {A}rithmetic {Q}uantum {U}nique {E}rgodicity},
    journal = {Annals of Mathematics},
    year = {2006},
    volume = {163},
    number = {2},
    pages = {165--219}
}

@article{SoundQUE,
    author = {Soundararajan, Kannan},
    title = {Quantum {U}nique {E}rgodicity for {$\mathrm{SL}_2(\mathbb{Z})\backslash \mathbb{H}$}},
    journal = {Annals of Mathematics},
    year = {2010},
    volume = {172},
    number = {2},
    pages = {1529--1538}
}

@article{ZelditchQES2,
    author = {Zelditch, Steven},
    title = {Quantum {E}rgodicity on the {S}phere},
    journal = {Commun. Math. Phys.},
    year = {1992},
    volume = {146},
    pages = {61--71}
}

@article{BSP,
    author = {Böcherer, Siegfried and Sarnak, Peter and Schulze-Pillot, Rainer},
    title = {Arithmetic and {E}quidistribution of {M}easures on the {S}phere},
    journal = {Commun. Math. Phys.},
    volume = {242},
    year = {2003},
    pages = {67--80}
}

@article{IchinoFormula,
    author = {Ichino, Atsushi},
    title = {Trilinear forms and the central values of triple product {$L$}-functions},
    journal = {Duke Math. J.},
    year = {2008},
    volume = {145},
    number = {2},
    pages = {281--307}
}

@article{Selberg1956, 
    title = {Harmonic Analysis and Discontinuous Groups in Weakly Symmetric Riemannian Spaces With Applications to Dirichlet Series},
    volume = {20},
    number = {1--3}, 
    journal = {J. Indian Math. Soc.}, 
    author = {Selberg, A.}, 
    year = {1956}, 
    pages = {47–-87}
}

@article{HumphriesRadziwill,
    author = {Humphries, Peter and Radziwiłł, Maksym},
    title = {Optimal Small Scale Equidistribution of Lattice Points on the Sphere, Heegner Points, and Closed Geodesics},
    journal = {Communications on Pure and Applied Mathematics},
    volume = {75},
    number = {9},
    pages = {1936--1996},
    doi = {https://doi.org/10.1002/cpa.22076},
    year = {2022}
}

@book{HallLieStuff,
    author = {Hall, Brian C.},
    title = {Lie {G}roups, {L}ie {A}lgebras, and {R}epresentations: {A}n {E}lementary {I}ntroduction},
    publisher = {Springer},
    year = {2015}
}

@book{AutRepGetz,
    author = {Getz, Jayce R. and Hahn, Heekyoung},
    title = {An {I}ntroduction to {A}utomorphic {R}epresentations},
    publisher = {Springer},
    year = {2024}
}

@book{QuantumTheoryforMath,
    author = {Hall, Brian C.},
    title = {Quantum {T}heory for {M}athematicians},
    publisher = {Springer},
    year = {2013}
}

@book{MorimotoSphere,
    author = {Morimoto, Mitsuo},
    title = {Analytic {F}unctionals on the {S}phere},
    publisher = {American Mathematical Society},
    series = {Translations of {M}athematical {M}onographs},
    volume = {178},
    year = {1998}
}

@article{CruzanProdsSpherHarms,
    author = {Cruzan, Orval R.},
    title = {Translational addition theorems for spherical vector wave functions},
    journal = {Quart. Appl. Math.},
    year = {1962},
    volume = {20},
    pages = {33--40}
}

@book{Voight,
    author = {Voight, John},
    title = {Quaternion Algebras},
    publisher = {Springer},
    year = {2021}
}

@article{AdjointBound,
    author = {Hoffstein, Jeffrey and Lockhart, Paul},
    title = {Coefficients of {M}aass {F}orms and the {S}iegel {Z}ero},
    journal = {Annals of Mathematics},
    year = {1994},
    volume = {140},
    number = {1}
}

@book{IwaniecKowalski,
    author = {Iwaniec, H. and Kowalski, E.},
    title = {Analytic {N}umber {T}heory},
    publisher = {American Mathematical Society},
    year = {2004}
}

@inproceedings{NTBTate,
    author = {Tate, J.},
    title = {Number {T}heoretic {B}ackground},
    booktitle = {Automorphic forms, representations, and {$L$}--functions, part {II}},
    editor = {Borel, A. and Casselman, W.},
    year = {1979},
    volume = {33},
    pages = {3--26}
}

@incollection{IntroLanglands,
    author = {Cogdell, J. W.},
    title = {Langlands {C}onjectures for {$\mathrm{GL}_n$}},
    booktitle = {An {I}ntroduction to the {L}anglands {P}rogram},
    editor = {Bernstein, J. and Gelbart, S.},
    publisher = {Springer},
    year = {2004},
    pages = {229--250}
}

@book{localLanglandsGL2,
    author = {Bushnell, C. and Henniart, G.},
    title = {The {L}ocal {L}anglands {C}onjecture for {$\mathrm{GL}(2)$}},
    publisher = {Springer},
    year = {2006}
}

@inproceedings{KnappLLCArchimedean,
    author = {Knapp, A. W.},
    title = {Local {L}anglands {C}orrespondence: {T}he {A}rchimedean {C}ase},
    booktitle = {Motives, {P}art {II}},
    year = {1994},
    editor = {Jannsen, Uwe and Kleiman, Steven and Serre, Jean-Pierre},
    publisher = {American Mathematical Society}
}

@book{GelbartAutFormsAdeleGps,
    author = {Gelbart, Stephen S.},
    title = {Automorphic {F}orms on {A}dele {G}roups},
    publisher = {Princeton University Press},
    year = {1975}
}

@incollection{BroConOes,
    author = {Conrad, Brian},
    title = {Reductive group schemes},
    booktitle = {Autour des sch{\'e}mas en groupes, {V}ol. {I}},
    publisher = {Soci{\'e}t{\'e} math{\'e}matique de {F}rance},
    year = {2014}
}

@article{HumphriesKhan,
    author = {Humphries, Peter and Khan, Rizwanur},
    title = {On the {R}andom {W}ave {C}onjecture for {D}ihedral {M}aa{\textbeta} {F}orms},
    journal = {Geom. Funct. Anal.},
    year = {2020},
    volume = {30},
    pages = {34--125}
}

@article{ShnirelmanQE,
    author = {Shnirel'man, A. I.},
    title = {Ergodic properties of eigenfunctions},
    journal = {Uspekhi Math. Nauk},
    year = {1974},
    volume = {29},
    number = {6},
    pages = {181--182}
}

@inbook{Shnirelman1993,
    author = {Shnirel'man, A. I.},
    title = {Addendum On the Asymptotic Properties of Eigenfunctions in the Regions of Chaotic Motion},
    booktitle = {KAM Theory and Semiclassical Approximations to Eigenfunctions},
    year = {1993},
    publisher = {Springer Berlin Heidelberg},
    pages = {313--337}
}

@article{deVerdiere85,
    author = {Colin de Verdière, Y.},
    title = {Ergodicit{\'e} et fonctions propres du laplacien},
    journal = {Comm. Math. Physics},
    year = {1985},
    volume = {102},
    pages = {497--502}
}

@article{RudnickSarnakQUE,
    author = {Rudnick, Ze{\'e}v and Sarnak, Peter},
    title = {{The behaviour of eigenstates of arithmetic hyperbolic manifolds}},
    volume = {161},
    journal = {Communications in Mathematical Physics},
    number = {1},
    pages = {195--213},
    year = {1994},
}

@article{HumphriesPhD,
    author = {Humphries, Peter},
    title = {Equidistribution in shrinking sets and {$L^4$}-norm bounds for automorphic forms},
    journal = {Mathematische Annalen},
    year = {2018},
    volume = {371},
    pages = {1497--1543}
}

@article{JakobsonQUEEisenstein,
     author = {Jakobson, Dmitry},
     title = {Quantum unique ergodicity for {Eisenstein} series on {$\mathrm{PSL}_2(\mathbb{Z}) \backslash \mathrm{PSL}_2(\mathbb{R})$}},
     journal = {Annales de l'Institut Fourier},
     pages = {1477--1504},
     volume = {44},
     number = {5},
     year = {1994}
}

@article{LuoSarnakQUE,
    author = {Luo, Wenzhi and Sarnak, Peter},
    title = {Quantum ergodicity of {E}igenfunctions on {$\mathrm{PSL}_{2}(\mathbb{Z}) \backslash \mathbb{H}^2$}},
    journal = {Publications math{\'e}matiques de l'I.H.{\'E}.S.},
    year = {1995},
    volume = {81},
    pages = {207--237}
}

@article{SarnakProgressQUE,
    author = {Sarnak, Peter},
    title = {Recent progress on the {Q}uantum {U}nique {E}rgodicity conjecture},
    journal = {Bull. of the American Math. Society},
    year = {2011},
    volume = {8},
    number = {2},
    pages = {211--228}
}

@book{ConwaySmith,
    author = {Conway, John H. and Smith, Derek A.},
    title = {On {Q}uaternions and {O}ctonions: {T}heir {G}eometry, {A}rithmetic, and {S}ymmetry},
    publisher = {A K Peters},
    year = {2003}
}

@book{Iwaniec,
    author = {Iwaniec, Henryk},
    title = {Topics in Classical Automorphic Forms},
    publisher = {American Mathematical Society},
    year = {1997}
}

@inproceedings{EichlerHeckeOp,
    author = {Eichler, M.},
    title = {The {B}asis {P}roblem for {M}odular {F}orms and the {T}races of the {H}ecke {O}perators},
    booktitle = {Modular Functions of One Variable I},
    editor = {Kuyk, Willem},
    year = {1972},
    pages = {75--152},
    publisher = {Springer}
}

@book{MontgomeryVaughan,
    author = {Montgomery, Hugh L. and Vaughan, Robert C.},
    title = {Multiplicative {N}umber {T}heory: {I}. {C}lassical {T}heory},
    publisher = {Cambridge University Press},
    year = {2006}
}

@article{HejhalRackner,
    author = {Hejhal, Dennis A. and Rackner, Barry N.},
    title = {On the {T}opography of {M}aass {W}aveforms for {PSL(2,Z)}},
    journal = {Experimental Mathematics},
    volume = {1},
    number = {4},
    pages = {275--305},
    year = {1992},
    publisher = {Taylor \& Francis},
    doi = {10.1080/10586458.1992.10504562}
}

@book{FollandHarmonic,
    author = {Folland, Gerald B.},
    title = {A Course in Abstract Harmonic Analysis},
    publisher = {CRC Press},
    edition = {{$2^{\circ}$}},
    year = {2015}
}

@misc{garza2025geodesicrestrictionproblemarithmetic,
    title = {The {G}eodesic {R}estriction {P}roblem for {A}rithmetic {S}pherical {H}armonics}, 
    author = {Sanchez Garza, Maximiliano},
    year = {2025},
    eprint = {2509.24874},
    archivePrefix = {arXiv},
    primaryClass = {math.NT},
    url = {\url{https://arxiv.org/abs/2509.24874}}, 
}

@article{LocalComponents,
    author = {Loeffler, David and Weinstein, Jared},
    title = {On the computation of local components of a newform},
    journal = {Math. Comp.},
    volume = {81},
    pages = {1179--1200},
    year = {2012}
}

@book{WeilAdelesAlgGroups,
    author = {Weil, Andr{\'e}},
    title = {Adeles and {A}lgebraic {G}roups},
    publisher = {Birkh{\"a}user},
    year = {1982},
    note = {With appendices by Michel Demazure and Takashi Ono}
}

@misc{Humphries2025,
    title = {Quantitative {E}quidistribution on {H}yperbolic {S}urfaces and {A}rithmetic {A}pplications}, 
    author = {Humphries, Peter},
    year = {2025},
    eprint = {2512.15664},
    archivePrefix = {arXiv},
    primaryClass = {math.NT},
    url = {\url{https://arxiv.org/abs/2512.15664}}, 
}

@book{SpheHar,
    author = {Atkinson, Kendall and Han, Weimin},
    title = {Spherical Harmonics and Approximations on the Unit Sphere: An Introduction},
    publisher = {Springer},
    year = {2012}
}

@book{VillaniOptTransport,
    title = {Topics in {O}ptimal {T}ransportation},
    author = {Villani, C{\'e}dric},
    publisher = {American Mathematical Society},
    year = {2003},
    series = {Graduate {S}tudies in {M}athematics 58}
}

@article{YoungQUEShrinking,
    author = {Young, Matthew P.},
    title = {The quantum unique ergodicity conjecture for thin sets},
    journal = {Advances in {M}athematics},
    year = {2016},
    volume = {286},
    pages = {958--1016}
}

@article{GrabnerTichy-Wasserstein,
    author = {Grabner, Peter J. and Tichy, Robert F.},
    journal = {Mathematics of Computation},
    number = {201},
    pages = {327--336},
    publisher = {American Mathematical Society},
    title = {Spherical {D}esigns, {D}iscrepancy and {N}umerical {I}ntegration},
    volume = {60},
    year = {1993}
}

@misc{BordaCuenin-Wasserstein,
    title = {Smoothing inequalities for transport metrics in compact spaces}, 
    author = {Borda, Bence and Cuenin, Jean-Claude},
    year = {2025},
    eprint = {2510.21380},
    archivePrefix = {arXiv},
    primaryClass = {math.CA},
    url = {https://arxiv.org/abs/2510.21380}, 
}

@article{Berry-Wasserstein,
    author = {Berry, Andrew C.},
    title = {The accuracy of the {G}aussian approximation to the sum of independent variates},
    journal = {Trans. Amer. Math. Soc.},
    year = {1941},
    volume = {49},
    pages = {122--136}
}

@article{Esseen-Wasserstein,
    author = {Esseen, Carl-Gustav},
    title = {Fourier analysis of distribution functions. {A} mathematical study of the {L}aplace--{G}aussian law},
    journal = {Acta Math.},
    year = {1945},
    volume = {77},
    pages = {1--125}
}

@article{BobkovLedoux-Wasserstein,
    author = {Bobkov, Sergey G. and Ledoux, Michel},
    title = {A {S}imple {F}ourier {A}nalytic {P}roof of the {AKT} {O}ptimal {M}atching {T}heorem},
    journal = {The Annals of Applied Probability},
    year = {2021},
    volume = {31},
    number = {6},
    pages = {2567--2584}
}

@article{Borda-Wasserstein,
    author = {Borda, Bence},
    title = {Equidistribution of {R}andom {W}alks on {C}ompact {G}roups {II}. {T}he {W}asserstein {M}etric},
    journal = {Journal of Fourier Analysis and Applications},
    year = {2021},
    volume = {27},
    number = {13}
}

@article{MartinDimensionsNew,
    author = {Martin, Greg},
    title = {Dimensions of the spaces of cusp forms and newforms on {$\Gamma_0(N)$} and {$\Gamma_1(N)$}},
    journal = {J. Number Theory},
    year = {2005},
    volume = {112},
    pages = {298--331}
}

@article{SpectrumLaplacian,
    author = {Schlichtkrull, Henrik and Trapa, Peter and Vogan, Jr., David A.},
    title = {Laplacians on {S}pheres},
    journal = {Sao Paulo J. Math. Sci.},
    year = {2018},
    volume = {12},
    number = {2},
    pages = {295--358}
}

@book{SteinWeiss,
    author = {Stein, Elias M. and Weiss, Guido},
    title = {Introduction to Fourier Analysis on Euclidean Spaces},
    publisher = {Princeton University Press},
    year = {1971}
}

@article{MinakshisundaramPleijelLocalWeylLaw, 
    title = {Some {P}roperties of the {E}igenfunctions of {T}he {L}aplace-{O}perator on {R}iemannian {M}anifolds}, 
    volume = {1}, 
    number = {3}, 
    journal = {Canadian Journal of Mathematics}, 
    author = {Minakshisundaram, S. and Pleijel, {\AA}.}, 
    year = {1949}, 
    pages = {242-–256}
}

@article{LubPhiSarI,
    author = {Lubotzky, A. and Phillips, R. and Sarnak, P.},
    title = {Hecke operators and distributing points on {$S^2$}. {I}},
    journal = {Communications on Pure and Applied Mathematics},
    volume = {39},
    number = {S1},
    pages = {S149--S186},
    year = {1986}
}

@article{VanderKam,
    author = {VanderKam, Jeffrey M.},
    title = {{$L^\infty$} {N}orms and {Q}uantum {E}rgodicity on the {S}phere},
    journal = {International Mathematics Research Notices},
    volume = {7},
    year = {1997},
    pages = {329--347}
}

@article{BLLQE,
    author = {Brooks, Shimon and Le Masson, Etienne and Lindenstrauss, Elon},
    title = {Quantum Ergodicity and Averaging Operators on the Sphere},
    journal = {International Mathematics Research Notices},
    volume = {2016},
    number = {19},
    pages = {6034--6064},
    year = {2015}
}

@article{HumphriesThorner,
    author = {Humphries, Peter and Thorner, Jesse},
    title = {New Variants of Arithmetic Quantum Ergodicity},
    journal = {Commun. Math. Phys.},
    year = {2025},
    volume = {406},
    number = {59}
}

@misc{WoodburyLocalConst,
    author = {Woodbury, Michael},
    title = {Explicit trilinear forms and subconvexity of the tiple product {$L$}-function},
    url = {\url{https://www.mi.uni-koeln.de/~woodbury/research/trilinear.pdf}},
    year = {2012}
}

\end{document}